\documentclass[12pt]{article}

\usepackage{amsmath, amsthm, amssymb, mathrsfs}
\usepackage{url}

\makeatletter \g@addto@macro\bfseries{\boldmath} \makeatother

\numberwithin{equation}{section}
\newtheorem{theo}{Theorem}[section]

\newtheorem{coro}[theo]{Corollary}
\newtheorem{lemm}[theo]{Lemma}

\def\aut{\mathrm{Aut}}
\def\atp{\mathrm{Atp}}
\def\ptp{\mathrm{Par}}
\def\grp{\mathcal{P}}
\def\id{\varepsilon}
\def\cyc{\mathscr{C}}

\renewcommand{\ge}{\geqslant}
\renewcommand{\le}{\leqslant}

\providecommand{\nequiv}{\not\equiv}

\def\tdot{{\cdot}}

\def\idl{\mathscr{D}}
\def\O{{\cal O}}
\def\sym{{\cal S}}
\def\Z{\mathbb{Z}}

\def\tref#1{Theorem~\ref{#1}}
\def\Tref#1{Table~\ref{#1}}
\def\lref#1{Lemma~\ref{#1}}
\def\cref#1{Corollary~\ref{#1}}

\def\sref#1{Section~\ref{#1}}

\DeclareMathOperator{\fix}{Fix}
\DeclareMathOperator{\lcm}{lcm}

\def\narrowcols{ \addtolength{\arraycolsep}{-1.5pt}\begin{footnotesize} }
\def\normalcols{ \end{footnotesize}\addtolength{\arraycolsep}{1.5pt} }
\def\verynarrowcols{ \addtolength{\arraycolsep}{-3pt}\renewcommand{\arraystretch}{0.8}\begin{scriptsize} }
\def\verynormalcols{ \end{scriptsize}\addtolength{\arraycolsep}{3pt}\renewcommand{\arraystretch}{1.25}}

\def\ttwo{\circ}
\def\tthree{\bullet}

\def\tspacer{{\vrule height 2.25ex width 0ex depth0ex}}
\def\bspacer{{\vrule height 0ex width 0ex depth1.0ex}}

\usepackage{enumitem}
\def\be{\begin{enumerate}[topsep=2pt,partopsep=0pt,itemsep=2pt,parsep=0pt]}

\begin{document}

\title{Autoparatopisms of Quasigroups and Latin Squares\footnote{Research supported by ARC grant FT100100153}}

\author{Mahamendige Jayama Lalani Mendis and Ian M.\ Wanless\\
\small School of Mathematical Sciences \\[-0.75ex]
\small Monash University\\[-0.75ex]
\small VIC 3800 Australia\\
\small {\tt\{jayama.mahamendige,ian.wanless\} @monash.edu} 
}

\date{}
\maketitle

\begin{abstract}
{\em Paratopism} is a well known action of the wreath product
$\sym_n\wr\sym_3$ on Latin squares of order $n$.  A paratopism that
maps a Latin square to itself is an {\em autoparatopism} of that Latin
square.  Let $\ptp(n)$ denote the set of paratopisms that are an
autoparatopism of at least one Latin square of order $n$.  We prove a
number of general properties of autoparatopisms. Applying these
results, we determine $\ptp(n)$ for $n\le17$. We also study the
proportion of all paratopisms that are in $\ptp(n)$ as
$n\rightarrow\infty$.

\end{abstract}

\section{Introduction}\label{s:intro}

Symmetry is one of the most important concepts in mathematics.  Latin
squares are two dimensional analogues of permutations that play a
pivotal role in areas as diverse as group theory, finite geometry,
statistical designs and coding theory, as well as in recreations such
as sudoku \cite{DKIII,LM98}. In this paper we investigate a
fundamental question: What symmetries can a Latin square have? See
\cite{Stones} for a survey of earlier results related to this topic,
stretching all the way back to Euler's seminal work. In particular, we
note that symmetry has played a critical role in enumerations such as
\cite{EW,HulpkeKaskiOstergard,McKayMeynertMyrvold2007} and is helpful
for creating Latin squares with desirable properties (see e.g.~the
survey \cite{diagcyc}). For work on computing the symmetries of a
Latin square, see \cite{Kot12,Kot14,McKayMeynertMyrvold2007}.

A Latin square of order $n$ is an $n \times n$ array containing $n$
symbols such that each symbol appears once in each row and each
column. Typically we will take $[n]=\{1,2,\dots,n\}$ to be the set of
symbols and also index rows and columns by the elements of $[n]$. The
element in the $i^{\rm th}$ row and $j^{\rm th}$ column of a Latin
square $L$ is denoted by $L(i,j)$. The set
$O(L)=\{(i,j,L(i,j)):i,j\in[n]\}$ of $n^2$ ordered triples is called
the \emph{orthogonal array representation of $L$}. The elements of
$O(L)$ will be called the \emph{triples} or \emph{entries} of $L$. A
\emph{quasigroup} $Q$ is a non-empty set together with a binary
operation `$\star$' such that for all $a, b \in Q$, there exist unique
$x, y \in Q$ satisfying $a \star x=b$ and $y \star a=b$. The operation
table of any quasigroup is a Latin square, and every Latin square can
be obtained in this way.
A simple but useful example is
the cyclic square $\cyc_n$, defined by $\cyc_n(i,j)\equiv i+j-1\mod n$,
which corresponds to the cyclic group. 

Let $\theta=(\alpha, \beta, \gamma) \in \sym_n^3$, where $\sym_n$ is
the symmetric group acting on $[n]$. A new Latin square $L^{\theta}$
is obtained by permuting rows, columns and symbols of a Latin square
$L$ by $\alpha$, $\beta$ and $\gamma$ respectively. That is, $L^{\theta}$ is
the Latin square defined by
$L^{\theta}(i,j)=L(i\alpha^{-1}, j\beta^{-1})\gamma$,
where we adopt the convention that permutations act from the right.
The map $\theta$ is known as an \emph{isotopism} and
$L^{\theta}$ is said to be \emph{isotopic} to $L$. If $L^{\theta}=L$,
then $\theta$ is called an \emph{autotopism} of $L$. Observe that
$\theta$ is an autotopism of $L$ if and only if
$L(i,j)\gamma=L(i\alpha,j\beta)$ for all $i,j \in [n]$. If
$\theta=(\alpha, \alpha, \alpha)$ and $L^{\theta}=L$, then $\alpha$ is
said to be an \emph{automorphism} of $L$, because it is an automorphism
of the associated quasigroup.

Define $\grp_n$ to be the wreath product $\sym_n\wr\sym_3$. We denote
a typical element $\sigma\in\grp_n$ by $(\alpha,\beta,\gamma;\delta)$
where $\theta=(\alpha,\beta,\gamma)\in\sym_n^3$ and $\delta\in\sym_3$.
This element acts on a Latin square $L$ of order $n$ to produce 
another Latin square $L^\sigma$ of the same order, where
$O(L^\sigma)$ is obtained by applying the permutation $\delta$
to the triples in $O(L^\theta)$.
For example, if $(x,y,z)$ is a triple of $L$ then
\[
(x,y,z)\sigma=
\begin{cases}
(y\beta, x\alpha, z\gamma), \text{ if }\delta=(12), \\
(z\gamma, x\alpha, y\beta), \text{ if }\delta=(123). 
\end{cases}
\]
The map $\sigma$ is known as a \emph{paratopism}. If $L^{\sigma}=L$,
then $\sigma$ is called an \emph{autoparatopism} of $L$. Throughout this
paper we use $\id$ to denote the identity permutation of the appropriate
degree. The group of isotopisms is a normal subgroup of $\grp_n$, corresponding
to the case when $\delta=\id$.

The groups of all automorphisms, autotopisms and autoparatopisms of
$L$ will be denoted by $\aut(L)$, $\atp(L)$ and $\ptp(L)$
respectively. We use $\aut(n)$ (respectively $\atp(n)$, $\ptp(n)$) 
to denote the set
of all elements of $\sym_n$ (respectively $\sym_n^3$, $\grp_n$) which
are an automorphism (respectively autotopism, autoparatopism) of
at least one Latin square of order $n$.

The primary aim of this paper is to better understand $\ptp(n)$. 
We establish a number of necessary conditions for $\sigma$ to be in
$\ptp(n)$. We find, by computation, that these necessary conditions
are sufficient when $n \le 17$. The analogous task for $\atp(n)$ was 
carried out in \cite{Stones} and our approach follows a
similar direction to that paper, at least initially. 
For an earlier catalogue of $\atp(n)$ for $n\le 11$ see \cite{Fal12} 
and for related work on partial Latin squares, see \cite{Fal13}.

Our paper is structured as follows. 
In \sref{s:tools} we introduce the basic tools, notation and terminology
with which we will study autoparatopisms.
In \sref{s:basic} we provide a number of general conditions that
autoparatopisms necessarily satisfy.
In \sref{s:12} and \sref{s:123} we consider those autoparatopisms
$(\alpha,\beta,\gamma;\delta)$ for which $\delta$ is respectively
a 2-cycle and a 3-cycle. Finally, in \sref{s:end} we draw the different
strands together. We determine $\ptp(n)$ for $n\le17$ and provide some
nice contrasts between autotopisms and more general autoparatopisms.

\section{Some basic tools and terminology}\label{s:tools}

\subsection{Cycle structures}

Every $\alpha\in \sym_n$ decomposes into a product of disjoint cycles,
where we consider fixed points to be cycles of length $1$.  
We denote the set of fixed points of $\alpha$ by $\fix(\alpha)$.
We say $\alpha$ has the \emph{cycle structure} $c_1^{\lambda_1}\tdot
c_2^{\lambda_2}\cdots c_m^{\lambda_m}$ if there are $\lambda_i$ cycles
of length $c_i$ in the unique cycle decomposition of $\alpha$ and
$c_1>c_2>\cdots>c_m\ge 1$.  Hence $n=\sum c_i\lambda_i$. If
$\lambda_i=1$, we may write $c_i$ instead of $c_i^1$ in the cycle
structure. If $i$ is a point moved by a particular cycle $C$ then we
say that $i$ is in $C$ and write $i\in C$. We use $o(C)$ to denote the
length of a cycle $C$ (in other words, the size of its orbit). We
write $o_\alpha(i)=c$ to denote that $i$ is in some cycle $C$ of the
permutation $\alpha$ for which $o(C)=c$.

A permutation in $\sym_n$ (and our expression of it)
is \emph{canonical} if (i) it is written as
a product of disjoint cycles, including 1-cycles corresponding to
fixed points, (ii) the cycles are ordered according to their length,
starting with the longest cycles, (iii) each $c$-cycle is of the form
$(i,i+1,\ldots,i+c-1)$, with $i$ being referred to as the {\em leading
symbol} of the cycle, and (iv) if a cycle with leading symbol $i$ is
followed by a cycle with leading symbol $j$, then $i < j$. For
each possible cycle structure there is precisely one canonical 
permutation with that cycle structure and there is a unique way
to represent it as a product of disjoint cycles.

The task of understanding $\ptp(n)$ is substantially simplified by
the following two results from \cite{MW15}.

\begin{lemm}\label{l:conjauto}
  Suppose $\sigma_1$ and $\sigma_2$ are conjugate in $\grp_n$. Then
  $\sigma_1 \in \ptp(n)$ if and only if $\sigma_2 \in \ptp(n)$.
\end{lemm}

We will use $\nu_1\sim\nu_2$ to denote that two
permutations $\nu_1,\nu_2$ are conjugate in $\sym_n$; in other words,
$\nu_1$ and $\nu_2$ have the same cycle structure. Note also that in
the next result we consider fixed points to be cycles (of length 1).

\begin{theo}\label{t:conj}
Suppose $\sigma_1=(\alpha_1,\alpha_2,\alpha_3;\delta_1)\in\grp_n$ and 
$\sigma_2=(\beta_1,\beta_2,\beta_3;\delta_2)\in\grp_n$. Then 
$\sigma_1$ is conjugate to $\sigma_2$ in $\grp_n$ if and only if
there is a length preserving
bijection $\eta$ from the cycles of $\delta_1$
to the cycles of $\delta_2$ with the following property:
If $\eta$ maps a cycle $(a_1\cdots a_k)$ to $(b_1\cdots b_k)$ 
then $\alpha_{a_1}\alpha_{a_2}\cdots\alpha_{a_k}\sim
\beta_{b_1}\beta_{b_2}\cdots\beta_{b_k}$.
\end{theo}

It follows from the above two results that any autoparatopism
$(\alpha,\beta,\gamma;\delta)$ is conjugate to an autoparatopism of
the form $(\alpha,\beta,\gamma;\id)$, $(\id,\beta,\gamma;(12))$ or
$(\id,\id,\gamma;(123))$. The first of these possibilities has
been well studied in \cite{Stones}, so we will concentrate mostly
on the second and third possibilities. Moreover, in these cases
the only salient consideration is the cycle structures of $\beta$
and $\gamma$. For this reason, we will often assume without
loss of generality that these permutations are canonical.

\subsection{Cell orbits}

We now discuss a notion that proved useful for studying $\atp(n)$ in
\cite{Stones}. In that work, the term {\em cell orbit} is used to
describe the set of cells of a Latin square in an orbit induced by an
autotopism. The concept is a useful one because the cell orbit is
determined by the autotopism and is independent of the contents of the
cells. The same property holds for autoparatopisms of the form
$(\alpha,\beta,\gamma;(12))$, so we also discuss cell orbits in this
case. Formally, suppose that
$\sigma=(\alpha,\beta,\gamma;(12))\in\ptp(L)$ for some Latin square
$L$. Then a cell orbit of $\sigma$ on $L$ is the projection onto the
first two coordinates of an orbit under the action of $\sigma$ on
$O(L)$. In most cases, $\sigma$ and $L$ will be implied by context
and we simply refer to ``cell orbits''.

\begin{lemm}\label{l:orb12}
  Suppose that $\sigma=(\id,\beta,\gamma;(12)) \in \ptp(L)$ for a
  Latin square $L$, where $\beta=\beta_1\beta_2\cdots\beta_p$ is a
  canonical permutation. Let $d_i=o_\beta(\beta_i)$ for $1\le i\le p$.
  Define $M_{ij}$ to be the $d_i\times d_j$ block of $L$ consisting of
  rows with indices in $\beta_i$ and columns with indices in $\beta_j$.

\be
\item[(i)] In each block $M_{tt}$ where $d_t$ is odd, there is one
  cell orbit with length $d_t$ and there are $(d_t -1) /2$ cell orbits
  with length $2d_t$.

\item[(ii)] In each block $M_{tt}$ where $d_t$ is even there are
  $d_t/2$ cell orbits with length $2d_t$.

\item[(iii)] If $s\ne t$ then there are
  $\gcd(d_s,d_t)$ cell orbits through $M_{st} \cup M_{ts}$, each with
  length $2\lcm(d_s,d_t)$. One half of each cell orbit lies in $M_{st}$
  and the other half lies in $M_{ts}$. 
\end{enumerate}
\end{lemm}

\begin{proof}
  Suppose $1\le s,t\le p$.
  The orbit of the cell $(i,j)$ in the block $M_{st}$ 
  can be divided into two subsets $A$ and $B$, where
\begin{align*}
  A &= \{(i\beta^l, j\beta^l): 1 \le l\le\lcm(d_s,d_t)\},\\
  B &= \{(j\beta^m, i\beta^{m-1}): 1\le m\le\lcm(d_s,d_t)\}.
\end{align*}
Observe that $|A|=\lcm(d_s,d_t)=|B|$.

First suppose that $s=t$. Then $A$, $B$ are subsets of the same
block, $M_{tt}$. Let $j=i\beta^q$ for some $q \in \{0,1,\ldots,
d_t-1\}$. Then $j\beta^m=i$
when $m=d_t-q$. So $i\beta^{m-1}=i\beta^{d_t-q-1}=j$ if
$d_t-q-1=q$, that is, when $q=(d_t-1)/2$. Such an integer
$q$ exists if and only if $d_t$ is odd. Hence $A$ and $B$ coincide in
the orbit of the cell $(i, i\beta^{(d_t-1)/2})$, so the length
of that orbit is $d_t$. All cells in 
$\{(i,i\beta^{(d_t-1)/2}):i\in\beta_t\}$ belong to the same cell orbit. 
The length of all other cell orbits is $2d_t$ when
$d_t$ is odd because $A$ and $B$ are disjoint. Similarly,
when $d_t$ is even, all cell orbits of $M_{tt}$ have length $2d_t$.

Finally, suppose that $s\ne t$. 
Then $A\subseteq M_{st}$ and $B\subseteq M_{ts}$
and hence the length of each cell orbit is $2\lcm(d_s, d_t)$.
\end{proof}

Example: Let $\sigma=(\id, \beta, \gamma;(12)) \in \grp_n$ where
$\beta$ is the canonical permutation with cycle structure $5^2$. Three
different cell orbits of $\sigma$ are represented by $\star$, $\ttwo$
and $\tthree$ in the following diagram.
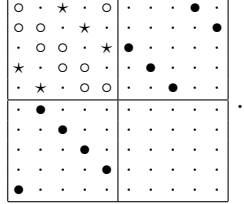
\[
\verynarrowcols
        \begin{array}{|ccccc|ccccc|}
        \hline
         \ttwo&\cdot&\star&\cdot&\ttwo&\cdot&\cdot&\cdot&\tthree&\cdot\\
        \ttwo&\ttwo&\cdot&\star&\cdot&\cdot&\cdot&\cdot&\cdot&\tthree\\
        \cdot&\ttwo&\ttwo&\cdot&\star&\tthree&\cdot&\cdot&\cdot&\cdot\\
        \star&\cdot&\ttwo&\ttwo&\cdot&\cdot&\tthree&\cdot&\cdot&\cdot\\
        \cdot&\star&\cdot&\ttwo&\ttwo&\cdot&\cdot&\tthree&\cdot&\cdot\\
        \hline
        \cdot&\tthree&\cdot&\cdot&\cdot&\cdot&\cdot&\cdot&\cdot&\cdot\\
        \cdot&\cdot&\tthree&\cdot&\cdot&\cdot&\cdot&\cdot&\cdot&\cdot\\
        \cdot&\cdot&\cdot&\tthree&\cdot&\cdot&\cdot&\cdot&\cdot&\cdot\\
        \cdot&\cdot&\cdot&\cdot&\tthree&\cdot&\cdot&\cdot&\cdot&\cdot\\
        \tthree&\cdot&\cdot&\cdot&\cdot&\cdot&\cdot&\cdot&\cdot&\cdot\\
        \hline
    \end{array}\ .
\endgroup\]

Next we consider autoparatopisms of the form
$(\id, \id, \gamma;(123))$.
In this case the term ``cell orbit'' is no longer appropriate since
all three coordinates in a triple affect its orbit. Hence we will
simply refer to orbits.

\begin{lemm}\label{l:orb123}
Suppose that $\sigma=(\id,\id,\gamma;(123)) \in \ptp(L)$ for a
Latin square $L$, where $\gamma=\gamma_1\gamma_2\cdots\gamma_p$ is a
canonical permutation. Let $d_i=o_\gamma(\gamma_i)$ for $1\le i\le p$. 
Define $M_{ij}$ to be the $d_i\times d_j$ block of $L$ consisting of rows 
with indices in $\gamma_i$ and columns with indices in $\gamma_j$.

Suppose $i\in\gamma_a$, $j\in\gamma_b$, $k\in\gamma_c$ and let $\O$ be
the orbit of the triple $(i,j,k)$ under $\sigma$.
If $\big|\{a,b,c\}\big| > 1$ then $\O$
has length $3\lcm(d_a,d_b,d_c)$, divided equally between three different blocks 
$M_{ab}$, $M_{ca}$, and $M_{bc}$. If $a=b=c$, then $\O$
lies entirely within $M_{aa}$. All such 
orbits have length $3d_a$ except that there may be one orbit of
length $d_a$ when $3\nmid d_a$.
\end{lemm}

\begin{proof}
  Let $(i, j, k)$ be a triple of $L$ such that $i, j, k$ are from
  cycles $\gamma_a, \gamma_b, \gamma_c$ respectively. The orbit
  of the triple $(i, j,k)$ of $L$ can be divided into $3$ subsets
  $A,B,C$, where
\begin{align*}
 A &= \{(i\gamma^l, j\gamma^l, k\gamma^l):1 \le l \le \lcm(d_a, d_b)\},
\\
 B &= \{(k\gamma^m, i\gamma^{m-1}, j\gamma^{m-1}):1 \le m \le \lcm(d_a, d_b)\},
\\
 C &= \{(j\gamma^r, k\gamma^r, i\gamma^{r-1}):1 \le r \le \lcm(d_a, d_b)\}.
\end{align*}
Observe that $|A|=|B|=|C|=\lcm(d_a,d_b)$ and
$A\subseteq M_{ab}$, 
$B\subseteq M_{ca}$, $C\subseteq M_{bc}$.
If $\big|\{a,b,c\}\big|>1$ then $M_{ab}$, $M_{ca}$ and $M_{bc}$ are
three different blocks. In this case 
the length of the orbit is $3\lcm(d_a,d_b)$.

Now suppose $i,j,k$ are from the same cycle of $\gamma$ and let the length
of that cycle be $d$. In this case $M_{ab}=M_{ca}=M_{bc}$. 
Viewing $A$, $B$, $C$ as orbits of $\sigma^3$ it becomes clear that either
they all coincide or they are pairwise disjoint.
If there exists an integer $m$ satisfying
$k\gamma^m=i$, $i\gamma^{m-1}=j$ and
$j\gamma^{m-1}=k$ then $i\gamma^{3m-2}=i$. If $3\mid d$ then $d\nmid 3m-2$
so $A,B,C$ are disjoint. If $3\nmid d$ there is a unique $m$ in the range
$1\le m\le d$ such that $d\mid 3m-2$. For that value of $m$, if
$(i,i\gamma^{m-1},i\gamma^{-m})$ is a triple of $L$ it will have an orbit
of length $d$. This orbit will contain all triples of the form 
$(i,i\gamma^{m-1},i\gamma^{-m})$ where $i\in\gamma_a$.
\end{proof} 

Example: Suppose $\sigma=(\id, \id, \gamma;(123)) \in \ptp(L)$ where
the cycle structure of $\gamma$ is $5^2\tdot1$. Two orbits, having 
respective lengths $5$ and $15$, are shown in the following figure.
\[
\narrowcols
        \begin{array}{|ccccc|ccccc|c|}
        \hline
        \cdot&\cdot&\cdot&2&\cdot&\cdot&\cdot&\cdot&\cdot&\cdot&\cdot\\
        \cdot&\cdot&\cdot&\cdot&3&\cdot&\cdot&\cdot&\cdot&\cdot&\cdot\\
        4&\cdot&\cdot&\cdot&\cdot&\cdot&\cdot&\cdot&\cdot&\cdot&\cdot\\
        \cdot&5&\cdot&\cdot&\cdot&\cdot&\cdot&\cdot&\cdot&\cdot&\cdot\\
        \cdot&\cdot&1&\cdot&\cdot&\cdot&\cdot&\cdot&\cdot&\cdot&\cdot\\
        \hline
        \cdot&\cdot&\cdot&\cdot&\cdot&\cdot&\cdot&\cdot&\cdot&\cdot&\cdot\\
        \cdot&\cdot&\cdot&\cdot&\cdot&\cdot&\cdot&\cdot&\cdot&\cdot&\cdot\\
        \cdot&\cdot&\cdot&\cdot&\cdot&\cdot&\cdot&\cdot&\cdot&\cdot&\cdot\\
        \cdot&\cdot&\cdot&\cdot&\cdot&\cdot&\cdot&\cdot&\cdot&\cdot&\cdot\\
        \cdot&\cdot&\cdot&\cdot&\cdot&\cdot&\cdot&\cdot&\cdot&\cdot&\cdot\\
        \hline 
        \cdot&\cdot&\cdot&\cdot&\cdot&\cdot&\cdot&\cdot&\cdot&\cdot&\cdot\\
        \hline
    \end{array}
\quad\quad
\begin{array}{|ccccc|ccccc|c|}
        \hline
        \cdot&\cdot&6&\cdot&\cdot&\cdot&\cdot&\cdot&3&\cdot&\cdot\\
        \cdot&\cdot&\cdot&7&\cdot&\cdot&\cdot&\cdot&\cdot&4&\cdot\\
        \cdot&\cdot&\cdot&\cdot&8&5&\cdot&\cdot&\cdot&\cdot&\cdot\\
        9&\cdot&\cdot&\cdot&\cdot&\cdot&1&\cdot&\cdot&\cdot&\cdot\\
        \cdot&10&\cdot&\cdot&\cdot&\cdot&\cdot&2&\cdot&\cdot&\cdot\\
        \hline
        \cdot&\cdot&\cdot&\cdot&2&\cdot&\cdot&\cdot&\cdot&\cdot&\cdot\\
        3&\cdot&\cdot&\cdot&\cdot&\cdot&\cdot&\cdot&\cdot&\cdot&\cdot\\
        \cdot&4&\cdot&\cdot&\cdot&\cdot&\cdot&\cdot&\cdot&\cdot&\cdot\\
        \cdot&\cdot&5&\cdot&\cdot&\cdot&\cdot&\cdot&\cdot&\cdot&\cdot\\
        \cdot&\cdot&\cdot&1&\cdot&\cdot&\cdot&\cdot&\cdot&\cdot&\cdot\\
        \hline 
        \cdot&\cdot&\cdot&\cdot&\cdot&\cdot&\cdot&\cdot&\cdot&\cdot&\cdot\\
        \hline
    \end{array}\ .
\endgroup\]

We call the cell orbits of length $d_t$ in \lref{l:orb12}(i) 
and the orbits of length $d_a$ in \lref{l:orb123}
{\em short orbits}. They will play a crucial
role in our subsequent work. 

\subsection{Contours}\label{Sc:Contours}

A contour is a tool introduced in \cite{Stones} to help define a
Latin square with a particular autotopism. 
A {\em contour} is a partial
Latin square containing exactly one filled cell in each cell orbit.
The whole
Latin square can then be recovered from knowledge of the autotopism.
We will find it convenient to employ this idea for autoparatopisms of
the form $(\alpha,\beta,\gamma;(12))$.
It is entirely routine to check that a contour works, that is,
produces a Latin square with the desired autoparatopism.  In all
cases, we leave this somewhat tedious checking to the reader when we
describe a contour.

When constructing contours we always assume that permutations are canonical.
Suppose that $\gamma=\gamma_1\cdots\gamma_c$.
For $1\le k\le c$, let $t_k=1+\sum_{j<k}o(\gamma_j)$ be the leading
symbol of $\gamma_k$. If $o(\gamma_k)>1$ we specify an orbit containing
symbols from $\gamma_k$ with the notation $C(i,j)=t_k$, which means that
the contour contains symbol $t_k$ in cell $(i,j)$. If $o(\gamma_k)=1$ we
instead write $C(i,j)=\infty_{k'}$ where $k'$ is used to index 
$\fix(\gamma)$. 

\section{General conditions}\label{s:basic}

In this section we determine some elementary necessary conditions 
that autoparatopisms must satisfy. We start by adapting several
results from \cite{Stones}, beginning with this lemma from that paper.

\begin{lemm}\label{l:lcmatp}
Let $\theta=(\alpha,\beta,\gamma)$ be an autotopism of a
Latin square $L$. If $o_\alpha(i)=a$ and $o_\beta(j)=b$, then
$o_\gamma(L(i,j))=c$, where
$\lcm(a,b)=\lcm(b,c)=\lcm(a,c)=\lcm(a,b,c)$.
\end{lemm}

Analogous conditions for autoparatopisms are as follows:

\begin{lemm}\label{l:lcm12}
Let $\sigma=(\alpha, \beta, \gamma;(12))$ be an autoparatopism of
a Latin square $L$. Suppose $(i,j,k)$ is a triple of $L$. 
If $o_{\alpha\beta}(i)=a$ and $o_{\beta\alpha}(j)=b$ and $o_\gamma(k)=c$ then
\[
\lcm(2a, 2b)=\lcm(2a, c)=\lcm(2b, c)=\lcm(2a, 2b, c).
\]
\end{lemm}

\begin{proof}
  It is clear that $(i, j, k){\sigma^{2p}}=(i(\alpha\beta)^p,
  j(\beta\alpha)^p, k\gamma^{2p})$, for any integer $p\ge0$.
  Therefore, 
  $(i, j, k){\sigma^{2\lcm(a, b)}}=(i,j,k\gamma^{2\lcm(a,b)})\in O(L)$.
  As $(i, j, k)\in O(L)$ we see that $k\gamma^{2\lcm(a,b)}=k$,
  so $c \mid 2\lcm(a,b)$. This
  means that $\lcm(2a, 2b)=\lcm(2a, 2b, c)$. Similarly,
\begin{align*}
 (i, j, k){\sigma^{\lcm(2a, c)}} 
&= (i(\alpha\beta)^{\lcm(2a, c)/2}, j(\beta\alpha)^{\lcm(2a, c)/2}, k\gamma^{\lcm(2a, c)})
= (i, j(\beta\alpha)^{\lcm(2a,c)/2}, k).
\end{align*}
Therefore $b \mid (\lcm(2a,c)/2)$. It follows that
$\lcm(2a,c)=\lcm(2a, 2b, c)$. A similar argument 
proves that $\lcm(2b,c)=\lcm(2a, 2b, c)$. 
 \end{proof}

\begin{lemm}\label{l:lcm123}
Let $\sigma=(\alpha, \beta, \gamma;(123))$ be an autoparatopism of
a Latin square $L$. Suppose $(i,j,k)$ is a triple of $L$. 
If $o_{\alpha\beta\gamma}(i)=a$,
$o_{\beta\gamma\alpha}(j)=b$ and $o_{\gamma\alpha\beta}(k)=c$, then
\[
\lcm(a, b)=\lcm(a, c)=\lcm(b, c)=\lcm(a, b, c).
\]
\end{lemm}

\begin{proof}
It is clear that $(i,j,k){\sigma^{3p}}=(i(\alpha\beta\gamma)^p,
j(\beta\gamma\alpha)^p, k(\gamma\alpha\beta)^p)$, for any integer $p\ge0$. 
Therefore
$(i,j,k){\sigma^{3\lcm(a,b)}}=(i,j,k(\gamma\alpha\beta)^{\lcm(a,b)})\in O(L)$. Hence
$c\mid \lcm(a,b)$, which means that $\lcm(a,b)=\lcm(a,b,c)$. Similarly we can
prove that $b \mid \lcm(a,c)$ and $a \mid \lcm(b,c)$. 
\end{proof} 

Let $\idl$ denote the set of ideals in the divisibility lattice of the
positive integers. In \cite{Stones} the elements of $\idl$ were called
{\em strongly lcm-closed sets}. This is because elements $\Lambda$ of
$\idl$ are characterised by the property that $a,b\in\Lambda$ if and
only if $\lcm(a,b)\in\Lambda$. If $\Lambda\in\idl$ is finite then
$\Lambda$ is the set of divisors of the maximum element in $\Lambda$.
Our next results are analogues of \cite[Thm~3.7]{Stones}.
A {\em subsquare} of a Latin square is a submatrix that is itself
a Latin square.

\begin{theo}\label{t:strlcm12}
Suppose $\Lambda\in\idl$. Let $2\Lambda=\{2x: x \in \Lambda \}$. Take
$\Lambda'=\Lambda \cup 2\Lambda$. (It can be proved that $\Lambda'\in\idl$). 
Suppose that
$\sigma=(\id, \beta, \gamma;(12))$ is an autoparatopism of a Latin
square $L$.
Define
$R_{\Lambda}=C_{\Lambda}=\{i\in[n]: o_\beta(i)\in \Lambda\}$, and
$S_{\Lambda'}=\{ i \in [n]: o_\gamma(i)\in \Lambda'\}$. 
If $R_{\Lambda}$ is
non-empty then $|R_{\Lambda}|=|C_{\Lambda}|=|S_{\Lambda'}|$ and $L$
contains a subsquare $M$ on the rows $R_{\Lambda}$, columns
$C_{\Lambda}$ and symbols $S_{\Lambda'}$.
\end{theo}

\begin{proof}
By definition $|R_{\Lambda}|=|C_{\Lambda}|$. Let $M$ be the submatrix
of $L$ induced by rows $R_{\Lambda}$ and columns $C_{\Lambda}$. Let
$(i,j,k)$ be any triple in $M$. Then $o_\beta(i)=a$ for
some $a \in \Lambda$, $o_\beta(j)=b$ for some $b\in\Lambda$ 
with $\lcm(2a, 2b)=\lcm(2a, c)=\lcm(2b, c)$, where $c=o_\gamma(k)$. 
Since $a,b\in \Lambda$ we know that $\lcm(a,b) \in \Lambda$. Therefore
$\lcm(2a,2b)=2\lcm(a, b) \in 2\Lambda$. Hence $\lcm (2a, c) \in
\Lambda'$. Therefore $c \in \Lambda'$ and hence
$|R_{\Lambda}|\le|S_{\Lambda'}|$.

Now consider any triple $(i,j,k)$ of $L$ for which
$o_\beta(i)=a \in \Lambda$ and $o_\beta(j)=d \notin \Lambda$. 
If $o_\gamma(k)=c$ then $\lcm(2a, 2d)=\lcm(2a,c)=\lcm(2d, c)$. 
Since $\lcm(a, d) \notin \Lambda$ we can be sure that $\lcm(2a,
2d)=2\lcm(a,d) \notin 2\Lambda$. Therefore $\lcm(2a, 2d) \notin
\Lambda'$. Hence $\lcm(2a, c) \notin \Lambda'$. But $2a \in
\Lambda'$, so $c \notin \Lambda'$. Hence $k \notin
S_{\Lambda'}$. Therefore in each row in $R_{\Lambda}$ the symbols
in $S_{\Lambda'}$ lie inside $M$. Therefore $M$ is a subsquare.
\end{proof}

For example, consider the paratopism $\sigma=(\id, \beta, \gamma;
(12))$ such that the cycle structure of $\beta$ is $18\tdot6\tdot2$
and the cycle structure of $\gamma$ is $18\tdot4^2$. Consider
$\Lambda=\{1,2\}\in\idl$. Then $\Lambda'=\Lambda \cup
2\Lambda=\{1,2,4\}$. Therefore $|R_{\Lambda}|=2$ and
$|S_{\Lambda'}|=8$. Hence by \tref{t:strlcm12}, $\sigma \notin \ptp(26)$.

We now show a more subtle use of \tref{t:strlcm12}. Suppose
$\sigma=(\id,\beta,\gamma;(12))\in\ptp(L)$ where the cycle structure
of $\beta$ is $4^2\tdot2^4$ and the cycle structure of $\gamma$ is
$8\tdot2\tdot1^6$.  Using $\Lambda=\{1,2\}$, we see that $L$ contains a
subsquare on the rows and columns indexed by the $2$-cycles of
$\beta$. As it happens, this subsquare can be constructed on the
symbols in $S_{\Lambda'}$. However, this subsquare is half the order
of $L$, which forces a subsquare, also on the symbols in
$S_{\Lambda'}$, to lie in the rows and columns indexed by the $4$-cycles
of $\beta$.  It will follow from \tref{t:fixgameven12} 
that this subsquare cannot be
built, and hence $\sigma\notin\ptp(16)$ after all.

Taking $\Lambda=\{1\}$ in \tref{t:strlcm12}, we immediately get:

\begin{coro}\label{cy:fixsbsq12}
Suppose $\sigma=(\id, \beta, \gamma;(12)) \in \ptp(L)$ for some
Latin square $L$ of order $n$.
If $1\le|\fix(\beta)|<n$ then the submatrix
$M$ whose rows and columns belong to $\fix(\beta)$ is a subsquare
of $L$, which means that $|\fix(\beta)| \le n/2$.
\end{coro}

\begin{theo}\label{t:strlcm123}
  Suppose that $\Lambda\in\idl$
  and that
  $\sigma=(\id, \id, \gamma; (123))$ is an autoparatopism of a Latin
  square $L$. Define,
  $R_{\Lambda}=\{i\in[n]:o_\gamma(i)\in \Lambda\}$. 
  If $R_{\Lambda}$ is non-empty then $L$ contains a
  subsquare $M$ on the rows $R_{\Lambda}$, columns $R_{\Lambda}$ and
  symbols $R_{\Lambda}$.
\end{theo}

\begin{proof}
Let $(i,j,k)$ be a triple of $L$, with $i,j\in R_{\Lambda}$. Then
$o_\gamma(i)=a$ and $o_\gamma(j)=b$ where $a, b \in \Lambda$. 
By \lref{l:lcm123}, $c=o_\gamma(k)$ satisfies 
$\lcm(a,c)=\lcm(b,c)=\lcm(a,b)\in \Lambda$, since $a,b\in\Lambda$. 
Therefore $c \in \Lambda$ and $k \in R_{\Lambda}$. 
Hence $M$ is a subsquare.
\end{proof}

Taking $\Lambda=\{1\}$ we immediately get:

\begin{coro}\label{cy:fixsbsq123}
Let $\sigma=(\id, \id, \gamma;(123)) \in \grp_n$ where $\gamma \neq\id$. 
If $\sigma \in \ptp(L)$ then the submatrix $M$ whose rows and
columns belong to $\fix(\gamma)$ is a subsquare of $L$, so
$|\fix(\gamma)| \le n/2$.
\end{coro}

We close this subsection by noting some results which are immediate
corollaries of prior work.

\begin{lemm}\label{l:totsym}
$(\id,\id,\id;\delta)\in\ptp(n)$ for all $\delta\in\sym_3$ and 
positive integers $n$. 
\end{lemm}

\begin{proof}
For all $n$ there is a totally-symmetric Latin square of order $n$, that is,
a Latin square whose set of triples is invariant under the natural action of
$\sym_3$.
For example, we can take $L(i,j)=-i-j\mod n$ with rows, columns and symbols
indexed by $\Z_n$.
\end{proof}

Consider the following result from \cite{Stones}.

\begin{theo}\label{t:pow2}
  Suppose that $2^a$ is the largest power of $2$ dividing $n$, where $a\ge1$. 
  Let $\theta=(\alpha, \beta, \gamma) \in \sym_n^3$, where
  the length of each cycle in $\alpha, \beta$ and $\gamma$ is
  divisible by $2^a$. Then $\theta \not\in \atp(n)$.
\end{theo}

This immediately implies:

\begin{coro}\label{cy:pow2para}
  Suppose that $2^{a}$ is the largest power of $2$ dividing $n$, where
  $a\ge1$. Let $\sigma=(\alpha, \beta, \gamma;(123))\in\grp_n$, where
  the length of each cycle in $\alpha\beta\gamma$ is divisible by
  $2^a$. Then $\sigma \notin \ptp(n)$.
\end{coro}

\begin{proof} 
  If $\sigma \in \ptp(L)$ then $\sigma^3=(\alpha\beta\gamma,
  \beta\gamma\alpha, \gamma\alpha\beta;\id) \in \atp(L)$. By
  assumption, the length of each cycle in $\alpha\beta\gamma$ is divisible 
  by $2^a$. However,
  $\alpha\beta\gamma\sim\beta\gamma\alpha\sim\gamma\alpha\beta$, so
  this is a contradiction of \tref{t:pow2}.
\end{proof}

For example,
$(\id,\id,\gamma;(123))\notin\ptp(n)$ when $\gamma$ has
cycle structure $d^m$, where $d$ is even and $m$ is odd.
If instead we make $m$ even in this example then 
\cref{cy:pow2para} tells us nothing, since then $n$ is divisible
by a higher power of 2 than $d$ is.

Bryant {\em et al.}~\cite{BryantBuchananWanless2009} considered Latin squares
with cyclic automorphisms and certain additional symmetries. Composing
the automorphism with an additional symmetry immediately gives the following:

\begin{theo}\label{t:fromBBW}
Suppose $\alpha \in \sym_n$ has cycle structure $(n-f)^1\tdot1^f$
Then $\sigma=(\alpha, \alpha, \alpha;(12))\in\ptp(n)$ if 
\be
\item[(i)] $f=0$ and $n$ is odd,
\item[(ii)] $f=1$, or 
\item[(iii)] $f=2$ and $n$ is even.
\end{enumerate}
Also $\sigma=(\alpha, \alpha, \alpha;(123)) \in \ptp(n)$ if 
\be
\item[(i)] $f=0$ and $n \equiv 1,3\mod 6$, 
\item[(ii)] $f=1$ and $n \not\equiv 0\mod6$ and $n\ne10$,
\item[(iii)] $f \equiv 2\mod3$ and $n \nequiv 0\mod3$,
\item[(iv)] $f \equiv 0\mod3, f \ge 3$ and $n \nequiv 2\mod3$, or
\item[(v)] $f \equiv 1\mod3$ and $f \ge 4$.
\end{enumerate}
\end{theo}

\section{Autoparatopisms of the form $(\alpha, \beta, \gamma; (12))$}\label{s:12}

In this section and the next we prove a number of general results
which together are sufficient to determine $\ptp(n)$ for $n\le17$.
By \tref{t:conj}, whether $\sigma=(\alpha,\beta,\gamma;(12))$ is in 
$\ptp(n)$ depends only on the cycle structure
of $\alpha\beta$ and the cycle structure of $\gamma$, so it is enough to
study paratopisms of the form $(\id, \beta, \gamma;(12))$. 
We start by proving a number of constraints on autoparatopisms of this
form. After that we study some special cases in which it is feasible 
to characterise exactly which paratopisms are autoparatopisms.

Many of the results in this section will employ the same basic
technique to bound the number of symbols which are fixed points of
$\gamma$. We concentrate on one row and consider how many columns in
that row may contain fixed symbols.  Another technique that we employ
repeatedly is to take a triple $T$, apply some power of a supposed
autoparatopism to produce a triple $T'$ that agrees in two places with $T$,
then deduce that $T'=T$. This relies on the fact that distinct triples
of a Latin square agree in at most one coordinate.

\begin{theo}\label{t:ord12}
Suppose that $(\id,\beta,\gamma;(12))\in\ptp(n)$. Let $b,c$ be respectively
the orders of $\beta,\gamma$ as elements of $\sym_n$. Then $c\mid2b$ and
if $c$ is odd then $c=b$. 
\end{theo}

\begin{proof}
Since $\sigma=(\id,\beta,\gamma;(12))\in\ptp(n)$ we know that
$\sigma^{2b}=(\id,\id,\gamma^{2b};\id)\in\ptp(n)$. It follows that
$\gamma^{2b}=\id$, so $c\mid2b$.

From now on suppose that $c$ is odd. Suppose that $d$ is any cycle length 
of $\beta$. It suffices to show that $d\mid c$. As $d$ was arbitrary 
this will show that $b\mid c$, which together with $c\mid 2b$ will imply
$b=c$, since $c$ is odd.

Consider a row $i$ such that $o_\beta(i)=d$ and let 
$k=L(i,i\beta^{(c-1)/2})$. As
$(i,i\beta^{(c-1)/2}, k)\in O(L)$, we know that
$O(L)$ also includes the triple
$$(i,i\beta^{(c-1)/2},k){\sigma^{c}}
=(i\beta^{c},i\beta^{(c-1)/2},k\gamma^{c})
=(i\beta^{c},i\beta^{(c-1)/2},k).$$
Hence, $i\beta^{c}=i$ which means that $d\mid c$, as required.
\end{proof}

\begin{coro}
$(\alpha,\beta,\id;(12))\in\ptp(n)$ if and only if 
$\alpha=\beta^{-1}\in\sym_n$.
\end{coro}

\begin{proof}
  We know that $(\alpha,\beta,\id;(12))\in\ptp(n)$
  if and only if $(\id,\alpha\beta,\id;(12))\in\ptp(n)$, by \tref{t:conj}.
If $\alpha=\beta^{-1}$ then 
$(\id,\alpha\beta,\id;(12))$ is in $\ptp(n)$ by \tref{l:totsym}.
\tref{t:ord12} shows that $(\id,\alpha\beta,\id;(12))\notin\ptp(n)$ if 
$\alpha\ne\beta^{-1}$.
\end{proof}

By \tref{t:ord12},
$(\id,\beta,\gamma;(12))\notin\ptp(12)$ if $\beta$ and $\gamma$
have cycle structures $6^2$ and $3^4$ respectively. Note
that \tref{t:ord12} does not eliminate the case when
$\beta$ and $\gamma$ have cycle structures $3^4$ and $6^2$ respectively. 
Our next result does rule out that case.

\begin{theo}\label{t:cycdiv12}
  Suppose $\sigma=(\id, \beta, \gamma;(12)) \in \ptp(n)$. 
  If $\beta$ has a cycle of odd length $d$ then $\gamma$ has at least one 
  cycle whose length divides $d$.
\end{theo}

\begin{proof}
Any symbol $k$ in the short orbit in the block with rows and columns 
indexed by the $d$-cycle of $\beta$ must satisfy $o_\gamma(k)\mid d$.
\end{proof}

In particular, \tref{t:cycdiv12} says that if $\beta$ has fixed points
then $\gamma$ has at least one fixed point. We next consider upper
bounds on the number of fixed points that $\gamma$ may have.

\begin{theo}\label{t:fixgameven12}
  Suppose that $\sigma=(\id, \beta, \gamma; (12)) \in\ptp(n)$.
  Fix a positive integer $d$ and let $r$ be the number of cycles of 
  $\beta$ that have length $d$. 
  If $r>0$ and $f$ is the number of fixed points of $\gamma$ then
\be
\item[(i)] $f\le(r-1)d$ if $d$ is even, and 
\item[(ii)] $f\le(r-1)d+1$ if $d$ is odd.
\end{enumerate}
\end{theo}

\begin{proof}
  We assume that $\sigma\in\ptp(L)$ for some Latin square $L$ of order
  $n$.  Fix a row $i$ such that $o_\beta(i)=d$ and suppose that
  $(i,j,k)\in O(L)$, where $o_\gamma(k)=1$. By \lref{l:lcm12},
  $o_\beta(j)=d$, giving us at most $rd$ options for $j$.  Now apply
  \lref{l:orb12}.  When $d$ is even, $j$ cannot be from the same orbit
  of $\beta$ as $i$ which means that $f\le(r-1)d$. When $d$ is odd,
  there is a unique possibility for $j$ in the same orbit of $\beta$
  as $i$, meaning that $f\le(r-1)d+1$.
\end{proof}

For example $\sigma=(\id, \beta, \gamma;(12)) \notin \ptp(n)$ when
$\beta$ and $\gamma$ have respective cycle structures
$4^2$ and $2^1\tdot1^6$. The same conclusion is reached 
if $\beta$ and $\gamma$ both have cycle structure
$3^2\tdot1^5$.

Our next result will improve on \tref{t:fixgameven12} in some cases
when $r$ is odd.

\begin{theo}\label{t:fixgamodd12}
Suppose that $\sigma=(\id, \beta, \gamma;(12)) \in \ptp(n)$.
Fix an integer $d>1$ and let $r$ be the number of cycles of 
$\beta$ that have length $d$.
Suppose that $r$ is odd.
Let $\Gamma$ be the set of all cycles $C$ of $\gamma$ satisfying
\be
\item[(i)] $o(C)$ is an odd divisor of $d$,
\item[(ii)] $\beta$ has no cycle $C'$ satisfying $1<o(C')<d$
and $\lcm(o(C),o(C'))=d$, and
\item[(iii)] $o(C)<d$ if $|\fix(\beta)|$ is odd.
\end{enumerate}
Then $\Gamma=\emptyset$ if $d$ is even and $|\Gamma|\le r$ if $d$ is odd.
\end{theo}

\begin{proof}
  Let $A$ be the submatrix of $L$ containing the cells $(i,j)$ for
  which $o_\beta(i)=o_\beta(j)=d$, and $B$ be the submatrix of $L$
  containing the cells $(i,j)$ for which $o_\beta(i)=d$ and $o_\beta(j)\ne d$.
  Assume that $C\in\Gamma$ and let $c=o(C)$.

  First suppose that some symbol $k$ of $C$ occurs in a column $j$ of
  $B$. Let $d'=o_\beta(j)$.  Since $c\mid d$ we know that
  $2d=\lcm(c,2d)=\lcm(c,2d')$ by \lref{l:lcm12}. So $d=\lcm(c,d')$, as
  $c$ is odd. Hence $d'=1$ by $(ii)$, which means that $d=\lcm(c,1)=c$.
  Now by $(iii)$, the number $f$ of fixed points of $\beta$ is even.  
  Suppose that $o_\beta(j')=1$. By \lref{l:lcm12}, each symbol of $C$
  occurs in column $j'$ of $B$, since it cannot occur in 
  column $j'$ of $L$ outside of $B$. 
  Hence the number
  of cells of $A$ that contain symbols from $C$ is $crd-fd$.
  These cells cannot be divided into orbits of length $2d$, 
  since $cr-f$ is odd. Hence, by \lref{l:orb12}, the symbols of $C$
  must fill at least one short orbit in $A$.
  
  Next suppose that no symbol of $C$ occurs in $B$.
  In that case the $crd$ cells of $A$ containing symbols of $C$ cannot
  be partitioned into orbits of length $2d$, since $cr$ is odd. So again,
  the symbols of $C$ must fill at least one short orbit in $A$.

  The number of short orbits in $A$ is $0$ if $d$ is even and $r$ if
  $d$ is odd. The result follows.
\end{proof} 

\tref{t:fixgameven12} implies that $|\fix(\gamma)|\le|\fix(\beta)|$.
\tref{t:fixgamodd12} provides another way to bound
the number of fixed points of $\gamma$,
since $\Gamma$ automatically contains all
such points, and often contains other cycles as well. For example,
if $\beta$ and $\gamma$ both have cycle structure $3^3\tdot1^2$ then
$(\id, \beta, \gamma;(12))\notin\ptp(11)$.

Our next result is a companion to \cref{cy:pow2para}.

\begin{theo}\label{t:evencyc12}
  Let $n=2^uv$ where $v$ is odd.
  Suppose that $\sigma=(\id, \beta, \gamma;(12)) \in \ptp(n)$ where
  every cycle of $\beta$ has length divisible by $2^u$. Then
\be
\item[(i)] $\beta$ has at least as many cycles of odd length as
  $\gamma$ has.
\item[(ii)] if $u\ge1$ then $\gamma$ has no cycles of odd length.
\end{enumerate}
\end{theo}

\begin{proof}
  Let $L$ be a Latin square for which $\sigma\in\ptp(L)$.
  Suppose that $(i,j,k)\in O(L)$ where $o_\gamma(k)=c$ and $c$ is odd.
  If $a=o_\beta(i)$ and $b=o_\beta(j)$
  then $2^{u+1}\mid 2\lcm(a,b)$ by assumption.
  Since $o_\gamma(k)=c$ and the symbol $k$ occurs $n$ times in $L$,
  the total length of all orbits containing $k$ is $cn$, which is not
  divisible by $2^{u+1}$. So $k$ must appear in at least one orbit whose
  length is not divisible by $2^{u+1}$.
  By \lref{l:orb12}, the only possibility is a short orbit.
  If $u\ge1$ then there are no short orbits. If $u=0$ then there is a unique
  short orbit for each cycle of $\beta$ of odd length. The result
  follows.
\end{proof}

For example, $(\id,\beta,\gamma;(12))\notin\ptp(n)$ if $\beta,\gamma$
have respective cycle structures $6^1\tdot2^2$ and $4^1\tdot3^2$; or
$3^3$ and $3^2\tdot1^3$.

To assist in the proof of our next theorem we introduce a well-known
concept.  For each pair $(s,t)$ of symbols in a Latin square $L$,
there are one or more {\em symbol cycles} which satisfying the
following description and are minimal in the sense that no proper
subset also satisfies the conditions. A symbol cycle  
is a set of cells that each contain either $s$ or $t$, 
and such that in any row or column of $L$ the symbol cycle
includes either zero or two cells. Symbol cycles are simple
examples of trades; a new Latin square can be obtained by 
switching the symbols $s$ and $t$ throughout the cycle
(see, for example, \cite{cycsw}). Below, we will also need to 
switch parts of symbol cycles. For this purpose we make two definitions.
We call the cells on the main diagonal of $L$ {\em pivots}. 
A {\em section} $S$ of a symbol cycle $C$ can then be defined as follows.
We designate a starting point in $C$, which will be a pivot and will
be included in $S$. We then alternate moving horizontally then vertically,
stepping to the other cell that $C$ contains in the same row or column,
respectively. Each cell that we visit is included in $S$. We stop immediately
upon including a second pivot in $S$.

\begin{theo}\label{t:betaid12}
  $\sigma=(\id,\id,\gamma;(12))\in\ptp(n)$ if and only if the cycle
  structure of $\gamma$ is $2^r\tdot 1^f$ for integers $r\ge0$ and $f\ge1$.
\end{theo}
 
\begin{proof}
Suppose that $\sigma=(\id,\id,\gamma;(12))\in\ptp(n)$. Then \tref{t:ord12}
shows that $\gamma^2=\id$ and \tref{t:cycdiv12} shows that $\gamma$
has a fixed point. Hence the cycle structure of $\gamma$ is as claimed.

Conversely, suppose that $\gamma$ has cycle structure $2^r\tdot 1^f$,
where $r\ge0$ and $f\ge1$. We now explain a process for constructing a
Latin square $L$ with $\sigma\in\ptp(L)$. Initially, we take $L$ to be
the cyclic square $\cyc_n$.  Then for $a=1,2,\ldots,r$ we undertake
the following steps, which we will describe as {\em surgery for $a$}.
We first identify the symbol cycles for the pair of symbols $(a,n-a)$.
These will be of two types depending on whether the cycle includes any
pivot or not. The symbol cycles that contain no pivot come in pairs
that are images of each other under transposition of $L$. We switch
the symbols throughout one entire cycle in each such pair of
cycles. Any symbol cycle $C$ that contains a pivot requires more
care. We first argue that $C$ contains exactly two pivots. It is clear
that the number of pivots in $C$ must be even since there are an even
number of cells in $C$ overall, and non-pivots are paired up by
transposition. Also, it is not hard to see that any section and its
image under transposition together form a complete symbol cycle.
Hence $C$ has two pivots as claimed. We will switch the symbols in one
section of $C$ and leave the rest of $C$ unaltered.  The details
depend on $n$, as follows.

Assume that $n$ is odd. Then
one pivot in $C$ contains the symbol $a$ and the other has the symbol $n-a$.
Suppose the former is cell $(i,i)$ and the latter is cell $(n+1-i,n+1-i)$.
We switch the section starting from $(i,i)$ and then put
$L(n+1-i,i)=a$ and $L(i,n+1-i)=n-a$.

Now assume that $n$ is even.
If $a$ is even then there are no relevant pivots, so assume $a$ is odd.
We switch the section starting at $((a+1)/2,(a+1)/2)$ and the section
starting at $(n-(a-1)/2,n-(a-1)/2)$. Then we put
$L((a+1)/2,n+1-(a+1)/2)=n-a$ and
$L(n+1-(a+1)/2,(a+1)/2)=a$. If $n\equiv0\!\mod4$ then put
$L((n+a+1)/2,(n-a+1)/2)=a$ and
$L((n-a+1)/2,(n+a+1)/2)=n-a$. Meanwhile, for
$n\equiv2\!\mod4$ we put
$L((n+a+1)/2,(n-a+1)/2)=n-a$ and
$L((n-a+1)/2,(n+a+1)/2)=a$.

For all values of $n$ the final step is to place the symbol $n$ in all
pivots of sections in which we have switched.
This completes our description of surgery for $a$.
It is now routine to check that it has the following properties. 
Surgery for $a$ arranges the symbols $a$ and $n-a$ in such a way
that if cell $(i,j)$ contains $a$ then cell $(j,i)$ contains $n-a$.
Moreover, the only places that $L$ changes are cells containing
$a$, $n-a$ or $n$. The last of these options only affects
cells in the row and column of the pivots that contain $a$ or $n-a$.
It follows that if $1\le a<a'\le r$ then surgery for $a$ affects 
cells that are distinct from those affected by surgery for $a'$.
Hence we can do surgery for each $a=1,2,\dots,r$ and the result
will be a Latin square having (a paratopism conjugate to)
$\sigma$ as an autoparatopism.
\end{proof}

We next present examples of the Latin squares constructed in
\tref{t:betaid12}, where the cycle structure of $\gamma$ is
$2^4\cdot 1^2$ in the left hand example and $2^3\cdot 1^5$ in the right
hand example.
\[
\narrowcols
\left[
\begin{array}{cccccccccc}
 10& 8& 3& 6& 5& 4& 7& 2& 1& 9\\ 
  2&10& 4& 5& 6& 3& 8& 9& 7& 1\\ 
  7& 6& 5& 4& 3& 2& 1&10& 9& 8\\ 
  4& 5& 6&10& 8& 9& 3& 1& 2& 7\\ 
  5& 4& 7& 2&10& 1& 9& 8& 3& 6\\ 
  6& 7& 8& 1& 9&10& 2& 3& 4& 5\\ 
  3& 2& 9& 7& 1& 8&10& 6& 5& 4\\ 
  8& 1&10& 9& 2& 7& 4& 5& 6& 3\\ 
  9& 3& 1& 8& 7& 6& 5& 4&10& 2\\ 
  1& 9& 2& 3& 4& 5& 6& 7& 8&10\\
\end{array}
\right]
\qquad
\left[
\begin{array}{ccccccccccc}
   11& 9& 3& 4& 5& 6& 7& 8& 2& 1&10\\
    2&11& 4& 5& 6& 7& 3& 9&10& 8&1\\
    8& 4& 5& 6& 7& 3& 9& 1&11&10&2\\
    4& 5& 6& 7& 3& 2&10&11& 1& 9&8\\
    5& 6& 7& 8&11& 1& 2&10& 9& 3&4\\
    6& 7& 8& 9&10&11& 1& 2& 3& 4&5\\
    7& 8& 2& 1& 9&10&11& 3& 4& 5&6\\
    3& 2&10&11& 1& 9& 8& 4& 5& 6&7\\
    9& 1&11&10& 2& 8& 4& 5& 6& 7&3\\
   10& 3& 1& 2& 8& 4& 5& 6& 7&11&9\\
    1&10& 9& 3& 4& 5& 6& 7& 8& 2&11\\
\end{array}\right].
\endgroup
\]

Squares having the symmetry discussed in \tref{t:betaid12} in the
particular case $f=1$ have been called {\em pairing squares}.
Some interesting properties of these squares were proven in 
\cite{atomic11} and \cite{diagcyc}.

In the remainder of this section we aim to build a number of Latin
squares with prescribed autoparatopisms. The target autoparatopism
will not be fully specified, but rather we will only know the cycle
structure of the permutations from which it is built. Without loss of
generality we will assume that these permutations are canonical to avoid
the need for caveats like the parenthetical phrase in the last sentence
of the proof of \tref{t:betaid12}.

\begin{theo}\label{t:bet1cyc12} 
  Suppose that $\sigma=(\id, \beta, \gamma;(12)) \in \grp_n$ and the
  cycle structure of $\beta$ is $n^1$. Then $\sigma \in \ptp(n)$ if
  and only if 
\be
\item[(i)] the length of each cycle in $\gamma$ divides $2n$ and
\item[(ii)] at most one cycle of $\gamma$ has odd length.
\end{enumerate}
\end{theo}

\begin{proof}
Suppose $\sigma \in \ptp(L)$ for some Latin square $L$  
and that $\gamma$ has a cycle of length $c$.  By \lref{l:orb12}, all
orbits have length $n$ or $2n$, so $c\mid 2n$.  If $c$ is odd then $c\mid n$.
By \tref{t:fixgamodd12}, there can be no such cycle when $n$ is even,
and at most one such cycle when $n$ is odd. Hence $(i)$ and $(ii)$ hold.

It remains to show sufficiency. Assume that $(i)$ and $(ii)$ hold.
Suppose that $\gamma$ has $c$ different cycles, with lengths
$d_1,\dots,d_c$. Since $n=\sum d_i$ we know from $(ii)$ that 
$\gamma$ has no cycles of odd length if $n$ is even.
Also, if $n$ is odd then $\gamma$ must have one 
cycle of odd length. For convenience, we assume that $d_1$ is odd
if $n$ is odd. Define $h=\lceil n/2\rceil+1$ and let 
$e_p=\sum_{j=1}^{p}\lceil d_j/2\rceil$ for $0\le p\le c$.
Then the contour of a Latin square $L$ such that $\sigma\in\ptp(L)$ can be 
constructed by putting
$C(i,h-i)=t_r$ for $1\le r \le c$ and $e_{r-1} < i \le e_r$.
\end{proof}

\begin{coro}\label{cy:dirpro}
  Suppose that $\sigma=(\id, \beta, \gamma;(12)) \in \grp_n$, the
  cycle structure of $\beta$ is $d^r$ and the cycle
  structure of $\gamma$ is $d_1^{ra_1}\tdot d_2^{ra_2}\cdots d_p^{ra_p}$.
  Then $\sigma \in \ptp(n)$ if $d_i\mid 2d$ for $1\le i\le p$, and
  at most one $d_i$ is odd.
\end{coro}

\begin{proof}
  \tref{t:bet1cyc12} provides a solution $L_1$ in the case $r=1$. For
  larger values of $r$, simply take the direct product of $L_1$ with $\cyc_r$.
\end{proof}

Our next result seems to depart from the principle of only considering
paratopisms of the form $(\id,\beta,\gamma;(12))$. However, it will
have several corollaries that deal with paratopisms of that form, so
can still be considered part of the same agenda.

\begin{theo}\label{t:1nontriv12}
  Suppose $\alpha \in \sym_n$ has cycle type $d^1\tdot 1^f$ where $d>1$. Let
  $\sigma=(\alpha, \alpha, \alpha;(12)) \in \grp_n$. Then
  $\sigma\in\ptp(n)$ if and only if one of the following is satisfied.
\be
\item[(i)] $d$ is odd and $f \in \{ 0, 1\}$, 
\item[(ii)] $d \equiv 0\mod4$ and $f \le d/2$, or
\item[(iii)] $d \equiv 2\mod4$ and $1 \le f \le d/2+1$.
\end{enumerate}
\end{theo}

\begin{proof} Throughout, $L$ will denote a hypothetical Latin square for
which $\sigma\in\ptp(L)$. We start by showing the necessity of conditions
$(i)$ to $(iii)$.
Fix a row $i$ such that $o_\alpha(i)=d$
and consider $(i,j,k)\in O(L)$ for which $o_\alpha(k)=1$.
By \lref{l:lcm12} we know that $o_\alpha(j)=d$, so $j$ and $i$ are from
the same orbit of $\alpha$. 

For the moment, assume that $d$ is odd. 
The orbit of the cell $(i,j)$ is
\[
\big\{(i\alpha^{2r}, j\alpha^{2r}):0\le r<d\big\}
\cup\big\{(j\alpha^{2r+1}, i\alpha^{2r +1}):0\le r<d\big\}.
\]
Since $d$ is odd there is an integer
$r$ such that $i=j\alpha^{2r+1}$. For this $r$ it must be the
case that $j=i\alpha^{2r+1}$, otherwise the symbol $k$ would occur
in two different cells in row $i$. It follows that $i=i\alpha^{4r+2}$,
so $4r+2$ is divisible by $d$. As $d$ is odd, it must divide
$2r+1$, but this means that $j=i\alpha^{2r+1}=i$. As there is only
one choice for $j$, we conclude that $f\le1$.

Next we assume that $d$ is even. In this case
the orbit of the cell $(i,j)$ is 
\[
\big\{(i\alpha^{2r}, j\alpha^{2r}):0\le r<d/2\big\} \cup
\big\{(j\alpha^{2r+1}, i\alpha^{2r +1}):0\le r<d/2\big\}.
\]
Suppose that $j=i\alpha^t$ for odd $t$ in the range $0<t<d$.
Then to avoid symbol $k$ being repeated in column $j$ we must have
$i=j\alpha^t=i\alpha^{2t}$. Hence $d$ divides $2t$, which can only mean
that $t=d/2$.
In other words $j\notin\{i\alpha^t:t=1,3,5,\dots,d-1\}\setminus\{i\alpha^{d/2}\}$
from which it follows that $f\le d/2$ when $d\equiv0\mod4$, 
and $f\le d/2+1$ when $d\equiv2\mod4$.

To complete the proof of necessity suppose that $d\equiv2\mod4$ and
consider the symbol $k'$ for which
$(i,i\alpha^{d/2},k')\in O(L)$. Applying $\sigma^{d/2}$ we find that
$(i,i\alpha^{d/2},k'\alpha^{d/2})\in O(L)$ which implies that
$o_\alpha(k')\mid (d/2)$. The only possibility is that $o_\alpha(k')=1$.
Hence $f\ge1$ when $d\equiv2\mod4$.

It remains to prove sufficiency of conditions $(i)$ to $(iii)$. For $(i)$ we
simply invoke \tref{t:fromBBW}. For $(ii)$
a contour for a Latin square $L$ such that
$\sigma\in\ptp(L)$ when $f \le d/2$ is as follows.
\begin{align*}
C(d-1,d) &= t_1,\\
C(i, d-i) &= t_1, &&\text{for } 1 \le i \le d/2-f, \\
C(d/2-f+i, d/2+f-i)&=\infty_i, &&\text{for } 1 \le i \le f, \\
C(d/2-f+i, \infty_i) = C(d+i, d/2+f-i) &=t_1, &&\text{for } 1 \le i \le f, \\
C(d/2+2i-1, d/2-2i) &=t_1, && \text{for } 1 \le i \le d/4-1, \\
C(d/2+2i, d/2+1-2i) &=t_1, && \text{for } 1 \le i \le d/4,
\end{align*}
together with any symmetric subquasigroup on the fixed points
of $\alpha$.

While for $(iii)$ we have the following contour.
Let $q=(d-2)/4$. Take,
\begin{align*}
C(q+1,d-q)=C(d-q,q+1)&=\infty_1,\\
C(d+1,q+1)=C(q+1,d+1)&=t_1,\\
C(i, d+1-i) = C(d+1-i, i) &=t_1, && \text{for } 1 \le i \le q, \\
C(q+i,d+i) = C(d+i, d-q+2-i) &=t_1, && \text{for } 2 \le i \le f, \\
C(q+i,d-q+2-i) &=\infty_i, && \text{for } 2 \le i \le f, \\
C(q+i,d-q+2-i) &=t_1, && \text{for } f+1 \le i \le d/2+1, 
\end{align*}
and add any symmetric subquasigroup on the fixed points of $\alpha$.
\end{proof}

Here are examples of the construction in \tref{t:1nontriv12} 
for $d=6$ and $f\in\{1,4\}$:
\[
\verynarrowcols
\left[
   \begin{array}{cccccc|c}
     4&2&5&\infty&3&1&6\\
     2&5&3&6&\infty&4&1\\
     5&3&6&4&1&\infty&2\\
     \infty&6&4&1&5&2&3\\
     3&\infty&1&5&2&6&4\\
     1&4&\infty&2&6&3&5\\
     \hline     \tspacer
     6&1&2&3&4&5&\infty\\ 
\end{array}
\right] 
\qquad  \qquad  
\left[
\begin{array}{cccccc|cccc}
  \infty_3&2&\infty_2&\infty_1&\infty_4&1&6&5&4&3\\
  2&\infty_3&3&\infty_4&\infty_1&\infty_2&1&6&5&4\\
  \infty_4&3&\infty_3&4&\infty_2&\infty_1&2&1&6&5\\
  \infty_1&\infty_2&4&\infty_3&5&\infty_4&3&2&1&6\\
  \infty_2&\infty_1&\infty_4&5&\infty_3&6&4&3&2&1\\
  1&\infty_4&\infty_1&\infty_2&6&\infty_3&5&4&3&2\\
  \hline
  \tspacer
  6&1&2&3&4&5&\infty_1&\infty_2&\infty_3&\infty_4\\
  3&4&5&6&1&2&\infty_2&\infty_3&\infty_4&\infty_1\\
  4&5&6&1&2&3&\infty_3&\infty_4&\infty_1&\infty_2\\
  5&6&1&2&3&4&\infty_4&\infty_1&\infty_2&\infty_3\\
\end{array} 
\right].
\endgroup
\]

\begin{coro}\label{k5}
Let $\sigma=(\id, \beta, \gamma;(12)) \in \grp_n$. Then $\sigma \in \ptp(n)$ if,
\be
\item[(i)] both $\beta$ and $\gamma$ have cycle structure $n^1$, where
  $n$ is odd,
\item[(ii)] both $\beta$ and $\gamma$ have cycle structure $(n-1)^1\tdot1^1$,
  where $n$ is even,
\item[(iii)] the cycle structure of $\beta$ is $(d/2)^2\tdot 1^f$ and
  the cycle structure of $\gamma$ is $d^1\tdot1^f$ where $d\equiv
  0\mod4$ and $f \le d/2$, or
\item[(iv)] the cycle structure of $\beta$ is $(d/2)^2\tdot1^f$ and
  the cycle structure of $\gamma$ is $d^1\tdot1^f$, where $d \equiv
  2\mod4$ and $1 \le f \le d/2 +1$.
\end{enumerate}
\end{coro}

\begin{proof}
Combine \tref{t:conj} and \tref{t:1nontriv12} when
$\beta\sim\alpha^2$ and $\gamma\sim\alpha$.
\end{proof}

\begin{coro}\label{cy:J10}
 Suppose $\alpha \in \sym_n$ has cycle structure $d^r$. Then 
 $\sigma=(\alpha, \alpha, \alpha;(12)) \in \ptp(n)$ if and only if 
 $d \nequiv 2\mod4$.
\end{coro}

\begin{proof}
We first consider the case when $d \equiv 2\mod4$.
Let $h=d/2$ and $k=L(1\alpha^h,1)$ where $L$ is a hypothetical
Latin square for which $\sigma\in\ptp(L)$. 
Then $(1\alpha^h,1,k){\sigma^h}=(1\alpha^h,1,k\alpha^h)$ since $h$ is odd, 
so $k\alpha^h=k$. But this contradicts $o_\alpha(k)=d$,
so $\sigma\notin\ptp(n)$.

Now suppose that $d\not\equiv2\mod4$. By \tref{t:1nontriv12} there
is a Latin square $L$ with $(\alpha',\alpha',\alpha';(12))\in\ptp(L)$
where $\alpha'$ has cycle structure $d^1$. The direct product
of $L$ and $\cyc_r$ has the required autoparatopism $\sigma$.
\end{proof}

\begin{coro}\label{cj:1cyc12} 
Let $\sigma=(\id, \beta, \beta;(12)) \in \grp_n$. Suppose $\beta$ has
cycle structure $d\tdot1^f$ where $d>1$. 
Then $\sigma \in \ptp(n)$ if and only if
one of the following conditions is satisfied: 
\be
\item[(i)] $d$ is odd and $f \in \{0, 1\}$, or

\item[(ii)] $d$ is even and $f=0$.
\end{enumerate}
\end{coro}

\begin{proof}
$(i)$ Suppose $d$ is odd. Then $(\beta, \beta, \beta;(12))$ and $(\id, \beta, \beta;(12))$ are conjugate, by \tref{t:conj}. 
By \tref{t:1nontriv12}, $\sigma \in \ptp(n)$ if and only if $f \in \{0,1\}$.

$(ii)$ Suppose $d$ is even. If $\sigma \in \ptp(n)$ then $f=0$ by
\tref{t:fixgamodd12}. Conversely, if $f=0$ then $\cyc_n$ has $\sigma$ as
an autoparatopism.
\end{proof}

Extending in the direction of \cref{cj:1cyc12}, we now characterise
$(\id, \beta, \beta;(12)) \in \ptp(n)$ when $\beta$ has only
two non-trivial cycles. We first do the case when those cycles 
are equal.

\begin{theo}\label{t:dsq12} 
Let $\sigma=(\id, \beta, \beta;(12)) \in \grp_n$. Suppose $\beta$ has
cycle structure $d^2\tdot1^f$ for some $d > 1$. Then
$\sigma\in\ptp(n)$ if and only if one of the following is satisfied.
\be
\item[(i)] $d$ is even and $f=0$, or
\item[(ii)] $d$ is odd and $f \le d+1$.
\end{enumerate}
\end{theo}

\begin{proof}
Suppose first that $\sigma \in \ptp(L)$ and $d$ is even. 
If $f\ge1$ then there is $(i,j,k)\in O(L)$ with $o_\beta(i)=d$
and $o_\beta(j)=1$. By \lref{l:lcm12}, $o_\beta(k)=2d$, but $\beta$
has no cycles of that length, so $f=0$. Conversely, if $f=0$ then
$\sigma \in \ptp(L)$ by \cref{cy:dirpro}.

Now suppose that $d$ is odd.  If $\sigma \in \ptp(L)$ 
then $f \le d+1$ by \tref{t:fixgameven12}.
Let $g=\lfloor f/2\rfloor$ and $h=(d+1)/2$. We construct a contour
for the subcase $f\le d$ first.
For $1\le i\le h$, take $C(i,h+1-i)=t_1$. For $1\le i\le g$ take
\begin{align*}
C(i+1,2d+i)=C(d+1-i,2d+g+i)=C(d+1+i,d+h-i)&=t_2\\
C(d+1+i,2d+i)=C(2d+1-i,2d+g+i)&=t_1\\
C(i+1,d+h-i)&=\infty_i\\
C(d+1-i,d+h+i)&=\infty_{i+g}.
\end{align*}
For $g+2\le i\le h$ take $C(i,d+h+1-i)=C(d+2-i,d+h-1+i)=t_2$ and
$C(d+i,d+h+1-i)=t_1$. If $f$ is odd then take $C(1,d+h)=\infty_f$,
$C(d+1,d+h)=C(1,n)=t_2$ and $C(d+1,n)=t_1$ whereas if $f$ is even,
take $C(1,d+h)=t_2$ and $C(d+1,d+h)=t_1$.

Finally, if $f=d+1$ then construct the contour for $f=d$ as above,
then vary it by taking $C(1,h)=C(d+1,d+h)=\infty_{d+1}$,
$C(1,n)=t_1$ and $C(d+1,n)=t_2$.

For $1\le f\le d+1$, we add any symmetric subquasigroup on
the fixed points of $\beta$.
\end{proof}

An example of the construction in \tref{t:dsq12} with $d = 5$ and $f = 1$:
\[
\verynarrowcols
\left[
\begin{array}{ccccc|ccccc|c} 
\tspacer
4& 5& 1& 2& 3& 9& 10& \infty& 7& 8& 6\\ 
5& 1& 2& 3& 4& 10& 6& 7& \infty& 9& 8\\ 
1& 2& 3& 4& 5& 6& 7& 8& 9& \infty& 10\\ 
2& 3& 4& 5& 1& \infty& 8& 9& 10& 6& 7\\ 
3& 4& 5& 1& 2& 8& \infty& 10& 6& 7& 9\\ 
\hline\tspacer 
9& 10& \infty& 7& 8& 4& 5& 6& 2& 3& 1\\ 
10& 6& 7& \infty& 9& 5& 1& 2& 8& 4& 3\\ 
6& 7& 8& 9& \infty& 1& 2& 3& 4& 10& 5\\ 
\infty& 8& 9& 10& 6& 7& 3& 4& 5& 1& 2\\ 
8& \infty& 10& 6& 7& 3& 9& 5& 1& 2& 4\\
\hline\tspacer
7& 9& 6& 8& 10& 2& 4& 1& 3& 5& \infty\\
\end{array}
\right].
\endgroup
\]

Next we look at $\beta$ with two non-trivial cycles of different lengths.

\begin{theo}\label{t:2diffcyc12} 
  Let $\sigma=(\id, \beta, \beta;(12)) \in \grp_n$. Let the cycle
  structure of $\beta$ be $d_1\tdot d_2 \tdot 1^{f}$, where
  $d_1>d_2>1$. Then $\sigma \in \ptp(n)$ if and only if $d_1/d_2$ is
  an odd integer and $f=0$.
\end{theo}

\begin{proof}
  Suppose $\sigma$ is an autoparatopism of a Latin square $L$. Let
  $(i,j,k)\in O(L)$ be such that $o_\beta(i)=d_1$ and $o_\beta(j)=d_2$.
  If $o_{\beta}(k)=c$, then by \lref{l:lcm12}, 
\[
\lcm(2d_1,2d_2)=\lcm(2d_1, c)=\lcm(2d_2, c).
\] 
This rules out $c\in\{1,d_2\}$ given that $d_1>d_2>1$.  Therefore
  $c=d_1$ and $\lcm(2d_1,2d_2)=2d_1$. 
  Hence $d_2 \mid d_1$. Now suppose that $d_1/d_2=2a$
  is even. Then
\begin{align*}
(i,j,k){\sigma^{2ad_2}} &= (i\beta^{ad_2}, j\beta^{ad_2}, k\beta^{2ad_2})
  = (i\beta^{ad_2}, j, k)
\end{align*}
But $i\beta^{ad_2} \neq i$ since $d_1 > ad_2$. This is a contradiction. 
Hence $d_1/d_2$ is odd. 

Suppose $f\ge1$ and apply \tref{t:fixgamodd12} with $d=d_1$.  The set
$\Gamma$ contains all fixed points of $\gamma$, so it cannot contain
the $d_2$-cycle. This must be because $d_2$ is even, from which we
conclude that $d_1$ is even, so $\Gamma=\emptyset$. Therefore $f=0$
after all.

Conversely suppose $d_1/d_2$ is an odd integer and $f=0$.  Taking
$\Lambda$ to be the set of divisors of $d_2$ in \tref{t:strlcm12}
shows that the $d_2$-cycle of $\beta$ induces a subsquare. This
subsquare can be built, by \cref{cy:dirpro}. For the remainder of the
contour, we consider three cases.

$(i)$ When $d_1$ and $d_2$ are even, take
\begin{align*}
C(i, d_1+1-i)  &= t_2, && \text{for } 1 \le i \le d_2/2, \\
C(i, d_1 +1-i) &= t_1, && \text{for } d_2/2+1 \le i \le d_1/2, \\
C (i, n+1-i) = 
C(d_1+i, d_1+1-i) &= t_1, &&\text{for } 1 \le i \le d_2/2. 
\end{align*}

$(ii)$ If $d_1 \equiv 1\mod4$, take
\begin{align*}
C(i, d_1+1-i) &=t_1, &&\text{for } (d_1+2d_2+5)/4 \le i \le (3d_1+1)/4, \\
C(i, d_1+1-i) &= t_2, &&\text{for } (d_1 +3)/4 \le i \le (d_1 +2d_2 +1)/4,  \\
C(i, (5d_1+2d_2+5)/4 -i) &= t_1, &&\text{for } (d_1-2d_2+5)/4 \le i \le (d_1+2d_2 +1)/4.
\end{align*}

$(iii)$ If $d_1 \equiv 3\mod4$, take
\begin{align*}
C(i, d_1+1-i) &= t_1, &&\text{for } (d_1+5)/4 \le i \le (3d_1-2d_2+1)/4, \\
C(i, d_1+1-i) &= t_2, &&\text{for } (3d_1-2d_2+5)/4 \le i \le (3d_1+3)/4,\\
C (i,(7d_1+2d_2+5)/4-i) &= t_1, &&\text{for } (3d_1-2d_2+5)/4 \le i \le (3d_1+2d_2+1)/4.
\end{align*}
\end{proof}

Our final result allows the shorter non-trivial cycle length of $\beta$
to be repeated.

\begin{theo}\label{t:k10}
  Let $\sigma=(\id, \beta, \beta;(12)) \in \grp_n$. Suppose that the cycle
  structure of $\beta$ is $d_1 \tdot {d_2}^l$, where $d_1$ is even
  and $d_1/d_2$ is an odd integer. 
  Then $\sigma\in\ptp(n)$ if and only if $0 \le l \le d_1/d_2$.
\end{theo}

\begin{proof}
Suppose that $L$ is such that $\sigma\in\ptp(L)$.
By \tref{t:strlcm12}, if $l>0$ then $L$ has a subsquare $S$ induced by the
cycles of length $d_2$. The order of $S$ is at most $n/2$,
which means that $ld_2\le d_1$. 

Conversely, suppose that $0 \le l \le d_1/d_2$. The subsquare $S$ can
be constructed by \cref{cy:dirpro}.
A contour for the rest of $L$ is as follows. We take
\begin{align*}
C(i, d_1+1-i) &= t_{k+1}, &&\text{for } 1 \le k \le l, \\
C(i, d_1+1-i) &= t_1,&&\text{for } l < k \le d_1/d_2,  \\
C(i, d_1+(3k-1)d_2/2+1-i) &= t_1, &&\text{for } 1 \le k \le l, \\
C(d_1+(k-1)d_2/2+i, d_1+1-i) &= t_1,  &&\text{for } 1 \le k \le l,
\end{align*}
for $(k-1) d_2/2 +1  \le i \le kd_2/2$.
\end{proof}

In this section we have demonstrated several conditions that
$(\id,\beta,\gamma;(12))\in\ptp(n)$ necessarily satisfy. In some of
the simpler subcases we were also able to provide sufficient conditions.
We have included these as examples of the types of results which may
be obtained.  However, given the complexities involved, we are not
optimistic that $(\id,\beta,\gamma;(12))\in\ptp(n)$ can be completely
characterised for general $n$.

\section{Autoparatopisms of the form $(\alpha, \beta, \gamma; (123))$}\label{s:123}

By \tref{t:conj}, whether $\sigma=(\alpha,\beta,\gamma;(123))$ is in 
$\ptp(n)$ depends only on the cycle structure
of $\alpha\beta\gamma$. Hence it is enough to
study paratopisms of the form $(\id, \id, \gamma;(123))$,
which is what we do in this section. The approach is very similar
to the previous section. We begin by proving some necessary conditions.

\begin{theo}\label{t:allindcyc}
Let $\sigma=(\id, \id, \gamma; (123)) \in \grp_n$.
Fix an integer $d$ and let $r$ be the number of cycles of 
$\gamma$ that have length $d$. 
Suppose that $\gamma$ has no two cycles of lengths $d',d''$ where
$d\notin\{d',d''\}$ and $\lcm(d,d')=\lcm(d,d'')=\lcm(d',d'')$.
Then $\sigma \not\in \ptp(n)$ if 
\be
\item[(i)] $r=1$ and $n+d\equiv1\mod3$, or
\item[(ii)] $3\mid d$ and $3 \nmid nr$.
\end{enumerate}
\end{theo}

\begin{proof}
The result is trivial if $r=0$, so assume $r\ge1$.
Suppose that $\sigma\in\ptp(L)$ for a Latin square $L$.
Define $\Omega=\{i\in[n]:o_\gamma(i)=d\}$.
Let $X$ be the submatrix of $L$ induced by the rows and columns
indexed by $\Omega$.
Suppose that $(i,j,k)\in O(L)$ where $i\in\Omega$ and 
$k\notin\Omega$. Then $j\in\Omega$ by \lref{l:lcm123} and our assumption on
cycle lengths of $\gamma$. In other words,
the $n-rd$ symbols that are not in $\Omega$ have
to occur in every row of $X$. This accounts for $rd(n-rd)$ of the $(rd)^2$
entries in $X$. The remaining entries will all be in orbits of length
$3d$ or in short orbits of length $d$. If $3\mid d$ then there are no
short orbits so we must have $3d\mid rd(2rd-n)$, which implies that
$3\mid rn$. On the other hand, if $r=1$ there is at most one short
orbit. In this case, either $3d\mid d(2d-n)$ or $3d\mid d(2d-n)-d$.
Both these conditions imply that $n+d\nequiv1\mod3$.
\end{proof}

\begin{theo}\label{t:eqcyc}
Suppose that $\sigma=(\id,\id,\gamma;(123))\in\ptp(n)$, where $\gamma$
has cycle structure $d^r$. Then 
\be
\item[(i)] if $3\mid d$ then $3\mid r$ and
\item[(ii)] if $6\mid d$ then $6\mid r$.
\end{enumerate}
\end{theo}

\begin{proof}
  Without loss of generality we assume that $\gamma$ is the canonical
  permutation with cycle structure $d^r$. Suppose there exists a
  Latin square $L$ of order $n$ such that $\sigma \in \ptp(L)$. 
  Define $\psi:O(L)\mapsto\Z_d$ by $\psi(i,j,k)\equiv j-i\mod d$.
  We assume throughout that $3\mid d$, so that each orbit of $\sigma$
  has length $3d$, by \lref{l:orb123}.
  Note that $\psi$ is constant on 
  orbits of $\sigma^3=(\gamma,\gamma,\gamma;\id)$, by our choice of $\gamma$.
Define $T$ to be the sum, modulo $d$, of $\psi$ over one representative
from each orbit of $\sigma^3$. Observe that
\[
\psi(i,j,k)+\psi(k\gamma,i,j)+\psi(j\gamma,k\gamma,i)
=j-j\gamma\equiv-1\mod d.
\]
Hence each orbit of $\sigma$ contributes $-1$ to $T$.
There are $n^2/(3d)=r^2d/3$ orbits of $\sigma$, so $T=-r^2d/3$. 
Counting the same quantity by taking $\psi$
of each triple in each row indexed by a multiple of $d$ we find that,
$$-\frac{r^2d}{3} \equiv r\sum_{i=1}^{n}i=\frac{rn(n+1)}{2} 
=\frac{r^2d(rd+1)}{2} \equiv 
\begin{cases} 
0,&\text{if $r$ is even or $d$ is odd,}\\
d/2,&\text{if $r$ is odd and $d$ is even,}    
\end{cases}$$
modulo $d$. Therefore, either $r^2/3$ or $r^2/3+1/2$ must be
an integer, but the latter option is impossible. We conclude that $3\mid r$
and either $r$ is even or $d$ is odd. The result follows.
\end{proof}

\tref{t:eqcyc} rules out several classes of autoparatopisms
$(\id,\id,\gamma;(123))$ where $\gamma$ is semi-regular (that is, all its
cycles have the same length). There is one
more case of non-existence where $\gamma$ is semi-regular, but it
seems to be isolated and not part of a family.

\begin{theo}\label{t:5sq}
If $\gamma$ has cycle structure $5^2$ then
$\sigma=(\id,\id,\gamma;(123))\notin\ptp(10)$.
\end{theo}

\begin{proof} 
Suppose $\sigma\in\ptp(L)$ and define blocks $M_{ij}$ as in \lref{l:orb123}. 
There are at most two short orbits of $\sigma$
and there are $20\equiv2\mod3$ orbits of $\sigma^3$, so both $M_{11}$ and
$M_{22}$ must contain a short orbit. 
The orbits of $\sigma$ that hit the $M_{12}$ block account for
$5$ of the remaining $8$ orbits of $\sigma^3$ in $M_{11}\cup M_{22}$.
Hence we may suppose without loss of generality that there is 
an orbit of $\sigma$ that is contained entirely within $M_{11}$.
This orbit must hit at least one of the cells $(1,1)$ or $(1,2)$
since it hits $3$ cells in the first row, and $(1,4)$ is in the short
orbit. However, a straightforward exhaustion of the possibilities
shows that no symbol is viable in either $(1,1)$ or $(1,2)$.
\end{proof}

Just as we did in the previous section we now seemingly depart from
our agenda in terms of the form of paratopisms we consider. However,
the result we prove will have corollaries relevant to our agenda.

\begin{theo}\label{t:aaa1nontriv123}
  Suppose $\sigma=(\alpha,\alpha,\alpha;(123)) \in \grp_n$, where
  $\alpha \in \sym_n$ has cycle structure $d^1\tdot1^f$. Then
  $\sigma\in\ptp(n)$ if and only if \be
\item[(i)] $f\equiv0\mod3$ and $d\nequiv2\mod3$,
with $d$ odd in the case $f=0$,
\item[(ii)] $f\equiv1\mod3$,
with $d\nequiv5\mod 6$ in the case $f=1$, or
\item[(iii)] $f\equiv2\mod3$ and $d\nequiv1\mod3$.
\end{enumerate}
\end{theo}

\begin{proof}
The following example has an autoparatopism
$(\alpha,\alpha,\alpha;(123))$ where $\alpha$ is the canonical
permutation with cycle structure $9^1\tdot1^1$.
\[
\verynarrowcols
\left[
\begin{array}{cccccccccc}
\tspacer
 8&10& 3& 7& 4& 5& 2& 1& 6& 9\\
 6& 9& 5& 1& 8& 2& 7& 3&10& 4\\
 3& 8& 7&10& 5& 9& 6& 4& 1& 2\\
 5& 4& 9& 2&10& 6& 1& 7& 8& 3\\
 1& 6&10& 9& 3& 8& 4& 2& 5& 7\\
 9& 7& 4& 6& 2& 1&10& 8& 3& 5\\
 4& 1& 2& 8& 7& 3& 5&10& 9& 6\\
 7& 5& 8& 4& 9&10& 3& 6& 2& 1\\
10& 2& 6& 3& 1& 7& 9& 5& 4& 8\\
 2& 3& 1& 5& 6& 4& 8& 9& 7&10\\
\end{array}
\right]
\endgroup
\]
\tref{t:fromBBW} shows $\sigma \in \ptp(n)$ in all other cases where
we are claiming existence.

Now suppose $\sigma \in \ptp(L)$. 
\cref{cy:pow2para} shows that $d$ must be odd when $f=0$.

For the remainder of the proof, assume that $3\nmid d$. 
Then, $(i,j,k){\sigma^d}=(k,i,j)$
when $d\equiv 1\mod3$ and $(i,j,k){\sigma^{2d}}=(j,k,i)$ when
$d\equiv 2\mod3$.  Hence, $L$ is semi-symmetric (that is, its
set of triples is invariant when the $3$ coordinates in each
triple are cyclically permuted). Also
\begin{align*}
(i,j,k){\sigma^{2d+1}}
&=(i\alpha, j\alpha, k\alpha) \in O(L) \text{ when } d\equiv 1\mod3,\\
(i,j,k){\sigma^{n+1}}
&=(i\alpha,j\alpha,k\alpha) \in O(L) \text{ when } d\equiv 2\mod3.
\end{align*}
Therefore, $\alpha$ is an automorphism of $L$. But, by 
\cite[Thrm~2.3]{BryantBuchananWanless2009}, $\alpha$ is not an automorphism 
of any semi-symmetric Latin square $L$ in the following cases: 
$d\equiv2\mod3$ and $f\equiv0\mod 3$, or
$d\equiv5\mod6$ and $f=1$, or
$d\equiv1\mod3$ and $f\equiv2\mod 3$.
\end{proof}

\begin{coro}\label{cy:gam1nontriv}
Let $\sigma=(\id,\id,\gamma;(123))\in\grp_n$, where $\gamma$ has cycle 
structure $d^1\tdot1^f$ and $3\nmid d$. Then $\sigma\in\ptp(n)$ if and
only if
\be
\item[(i)] $f\equiv0\mod3$ and $d\equiv1\mod3$,
with $d$ odd in the case $f=0$,
\item[(ii)] $f\equiv1\mod3$,
with $d\nequiv5\mod 6$ in the case $f=1$, or
\item[(iii)] $f\equiv d\equiv2\mod3$.
\end{enumerate}
\end{coro}

\begin{proof}
As $3\nmid d$ we see that $\gamma\sim\gamma^3$ so
$\sigma$ is conjugate to $(\gamma,\gamma,\gamma;(123))$ in
$\grp_n$, by \tref{t:conj}. 
The result now follows from \tref{t:aaa1nontriv123}.
\end{proof}

\begin{coro}\label{cy:singcyc}
Let $\sigma=(\id, \id, \gamma;(123)) \in \grp_n$, where $\gamma$ has cycle 
structure $n^1$. Then $\sigma\in\ptp(n)$ if and only if $n \equiv 1\mod6$.
\end{coro}

\begin{proof}
  If $3\mid n$ then \tref{t:eqcyc} shows that $\sigma\notin\ptp(n)$.
  If $3\nmid n$ then we apply \cref{cy:gam1nontriv}.
\end{proof}

\begin{coro}\label{cy:3cycs}
Let $\sigma=(\id, \id, \gamma;(123))\in\grp_n$ where $\gamma$ has 
cycle structure $d^3$. Then $\sigma\in\ptp(n)$ if and only if $d$
is odd.
\end{coro}

\begin{proof}
Let $\alpha\in\sym_n$ be a single cycle of length $3d$ so that
$\alpha^3$ has cycle structure $d^3$. Then 
$\sigma'=(\alpha,\alpha,\alpha;(123))$ 
is conjugate to $\sigma$ in $\grp_n$, by 
\tref{t:conj}. Now apply \tref{t:aaa1nontriv123}$(i)$.
\end{proof}

\begin{table}
\[
\begin{array}{ccc}
\verynarrowcols\begin{array}{|r|r|}
\hline
\tspacer n=2 & \\
\beta & \gamma\bspacer\\
\hline\tspacer
1^2 & 1^2 \\
2 & 2 \\
\hline
\hline
\tspacer n=3 & \\
\beta & \gamma\bspacer\\
\hline\tspacer
1^3 & 1^3, 2\tdot 1 \\
3 & 2 \tdot 1,3 \\

\hline
\hline
\tspacer n=4 & \\
\beta & \gamma\bspacer\\
\hline\tspacer
1^4 & 1^4,2\tdot 1^2 \\
2^2 & 2\tdot 1^2,2^2,4 \\
3\tdot 1 & 3\tdot 1 \\
4 & 2^2,4 \\
\hline
\hline
\tspacer n=5 & \\
\beta & \gamma\bspacer\\
\hline\tspacer
1^5 & 1^5,2\tdot1^3,2^2\tdot1 \\
2^2\tdot 1 & 4 \tdot 1 \\
5 & 2^2\tdot1, 5 \\
\hline
\hline
\tspacer n=6 & \\
\beta & \gamma\bspacer\\
\hline\tspacer
1^6 & 1^6,2\tdot1^4,2^2\tdot1^2 \\
2^2\tdot1^2 & 4\tdot1^2 \\
2^3 & 2^3, 4\tdot 2 \\
3\tdot 1^3 & 3\tdot 2\tdot 1 \\
3^2 & 2 \tdot 1^4, 2^2 \tdot 1^2, 3 \tdot 1^3, 3 \tdot 2 \tdot 1, 3^2 \\
5 \tdot 1 & 5 \tdot 1 \\
6 & 2^3, 4\tdot 2, 6 \\
\hline
\hline
\tspacer n=7 & \\
\beta & \gamma\bspacer\\
\hline\tspacer
1^7 & 1^7,2\tdot1^5,2^2\tdot1^3,2^3\tdot1 \\
2^2 \tdot 1^3 & 4 \tdot 2 \tdot 1 \\
3^2 \tdot 1 & 3^2\tdot 1,6 \tdot 1 \\
7 & 2^3 \tdot 1, 7\\
\hline
\hline
\tspacer n=8 & \\
\beta & \gamma\bspacer\\
\hline\tspacer
1^8 & 1^8,2\tdot1^6,2^2\tdot1^4,2^3\tdot1^2 \\
2^2\tdot1^4 & 4\tdot2\tdot1^2 \\
2^4 & 2\tdot1^6, 2^2\tdot1^4, 2^3\tdot1^2, 2^4, \\
& 4\tdot1^4, 4\tdot2\tdot1^2, 4\tdot2^2, 4^2 \\        
3^2\tdot1^2 & 3^2\tdot1^2, 6\tdot1^2  \\
4^2 & 2^2\tdot1^4, 2^3\tdot1^2, 2^4, 4\tdot1^4, \\
& 4\tdot2\tdot1^2, 4\tdot2^2, 4^2, 8 \\
5\tdot1^3 & 5\tdot2\tdot1 \\
6\tdot2 & 3^2\tdot 2,6\tdot2 \\
7\tdot1 & 7\tdot1 \\
8 & 2^4, 4\tdot2^2, 4^2, 8 \\
\hline
\hline
\tspacer n=9 & \\
\beta & \gamma\bspacer\\
\hline\tspacer
1^9 & 1^9, 2\tdot1^7, 2^2\tdot1^5, 2^3\tdot1^3, 2^4 \tdot1 \\
2^4\tdot 1 & 4^2\tdot 1 \\
3^2\tdot 1^3 & 3^2\tdot 1^3, 3^2\tdot 2\tdot 1, 6 \tdot1^3, 6\tdot 2 \tdot1  \\
3^3 & 2^3\tdot 1^3, 2^4\tdot 1, 3 \tdot2^2\tdot 1^2, 3\tdot 2^3, \\
    &  3^2\tdot 2\tdot 1, 3^3, 6 \tdot1^3, 6\tdot 2\tdot 1, 6\tdot 3  \\
4^2\tdot1 & 8\tdot1 \\
9 & 2^4\tdot 1,3\tdot2^3, 6\tdot 2\tdot 1,  6\tdot3,9\\
\hline
\end{array}
\verynormalcols 

&&

\verynarrowcols\begin{array}{|r|r|}
\hline
\tspacer n=10 & \\
\beta & \gamma\bspacer\\
\hline\tspacer
1^{10} & 1^{10}, 2\tdot1^8, 2^2\tdot1^6, 2^3\tdot1^4, 2^4\tdot1^2\\
2^4\tdot1^2 & 4^2\tdot1^2 \\
2^5 & 2^5, 4\tdot2^3, 4^2\tdot2 \\
3^2\tdot1^4 & 3^2\tdot1^4, 3^2\tdot2\tdot1^2,  6\tdot1^4, 6\tdot2\tdot1^2\\
3^3\tdot1 & 3^3\tdot1, 6\tdot3\tdot1 \\
4^2\tdot1^2 & 8\tdot1^2\\
4^2\tdot2 & 8\tdot2\\
5\tdot1^5 & 5\tdot2^2\tdot1 \\
5^2 & 2^2\tdot1^6, 2^3\tdot1^4, 2^4\tdot1^2, 5\tdot1^5, 5\tdot2\tdot1^3, 5\tdot2^2\tdot1, 5^2 \\
6\tdot2^2 & 6\tdot2^2, 6\tdot4\\
7\tdot1^3 & 7\tdot2\tdot1 \\
9\tdot1 & 9\tdot1 \\
10 & 2^5, 4\tdot2^3, 4^2\tdot2,10 \\
\hline
\hline
\tspacer n=11 & \\
\beta & \gamma\bspacer\\
\hline\tspacer
1^{11} & 1^{11}, 2\tdot1^9, 2^2\tdot1^7, 2^3\tdot1^5, 2^4\tdot1^3, 2^5\tdot1 \\
2^4\tdot1^3 & 4^2\tdot 1^3, 4^2\tdot 2\tdot 1 \\
3^2\tdot1^5 & 3^2\tdot 2\tdot 1^3, 3^2\tdot 2^2\tdot 1, 6\tdot 2 \tdot1^3, 6\tdot 2^2\tdot 1 \\
3^3\tdot 1^2 & 6\tdot3\tdot 1^2\\
4^2\tdot1^3 & 8\tdot1^3, 8\tdot2\tdot1 \\
5^2\tdot1 & 5^2\tdot1, 10 \tdot1 \\
11 & 2^5\tdot1, 11 \\
\hline
\hline
\tspacer n=12 & \\
\beta & \gamma\bspacer\\
\hline\tspacer
1^{12} & 1^{12}, 2\tdot1^{10}, 2^2\tdot1^8, 2^3\tdot1^6, 2^4\tdot1^4, 2^5\tdot1^2 \\
2^4\tdot1^4 & 4^2\tdot1^4, 4^2\tdot2\tdot1^2 \\
2^6 & 2\tdot1^{10}, 2^2\tdot1^8, 2^3\tdot1^6, 2^4\tdot1^4, 2^5\tdot1^2, 2^6, 4\tdot1^8,4\tdot2\tdot1^6, \\
& 4\tdot2^2\tdot1^4, 4\tdot2^3\tdot1^2, 4\tdot2^4, 4^2\tdot1^4, 4^2\tdot2\tdot1^2, 4^2\tdot2^2, 4^3 \\
3^2\tdot1^6 & 3^2\tdot2\tdot1^4, 3^2\tdot2^2\tdot1^2, 6\tdot2\tdot1^4, 6\tdot2^2\tdot1^2\\
3^3\tdot1^3 & 3^3\tdot1^3, 3^3\tdot2\tdot1, 6\tdot3\tdot1^3, 6\tdot3\tdot2\tdot1\\
3^4 & 2\tdot1^{10}, 2^2\tdot1^8, 2^3\tdot1^6, 2^4\tdot1^4, 2^5\tdot1^2, \\
& 3\tdot1^9, 3\tdot2\tdot1^7, 3\tdot2^2\tdot1^5, 3\tdot2^3\tdot1^3, 3\tdot2^4\tdot1, \\
& 3^2\tdot 1^6, 3^2\tdot 2\tdot 1^4, 3^2 \tdot2^2\tdot 1^2, 3^2\tdot 2^3, 3^3\tdot 1^3, 3^3\tdot 2\tdot 1, 3^4, \\
& 6\tdot1^6, 6\tdot2\tdot1^4, 6\tdot2^2\tdot1^2,6\tdot3\tdot1^3, 6\tdot3\tdot2\tdot1, 6\tdot3^2 \\
4^2\tdot1^4 & 8\tdot1^4, 8\tdot2\tdot1^2 \\
4^2\tdot 2^2 & 8 \tdot2\tdot 1^2, 8\tdot 2^2, 8\tdot 4 \\
4^3 & 2^6, 4\tdot2^4, 4^2\tdot 2^2, 4^3, 8\tdot 2^2, 8\tdot 4 \\
5^2\tdot 1^2 & 5^2\tdot 1^2, 10\tdot 1^2 \\
6\tdot 2^3 & 3^2\tdot 2^3, 4\tdot 3^2\tdot 2, 6\tdot2^3, 6\tdot 4\tdot 2\\
6^2 &  2^3\tdot 1^6, 2^4 \tdot1^4, 2^5 \tdot1^2, 2^6, 
  3 \tdot2^2 \tdot1^5, 3\tdot 2^3\tdot 1^3, 3\tdot 2^4\tdot 1, \\
& 3^2\tdot 2\tdot 1^4, 3^2\tdot 2^2 \tdot1^2, 3^2\tdot 2^3, 3^3\tdot 2\tdot 1, 4\tdot 2\tdot 1^6, \\
& 4 \tdot2^2\tdot 1^4, 4\tdot 2^3 \tdot1^2, 4\tdot 2^4, 4\tdot 3\tdot 1^5, 4 \tdot3 \tdot2\tdot 1^3, 4\tdot 3\tdot 2^2\tdot 1,\\
& 4 \tdot3^2\tdot 1^2, 4\tdot 3^2\tdot 2, 4^2\tdot 1^4,
4^2\tdot 2\tdot 1^2, 4^2 \tdot2^2, 4^2 \tdot3 \tdot1, 4^3, \\ 
& 6\tdot 1^6, 6\tdot 2\tdot 1^4, 6 \tdot2^2 \tdot1^2, 6\tdot 2^3, 6\tdot 3 \tdot1^3, 6 \tdot3 \tdot2 \tdot1, 6\tdot 3^2,\\
&6\tdot 4\tdot 1^2, 6 \tdot4\tdot 2, 6^2, 12\\
7\tdot1^5 & 7\tdot 2^2\tdot 1 \\
9\tdot 1^3 & 9\tdot 2\tdot 1 \\
9\tdot 3 & 9 \tdot2 \tdot1, 9\tdot 3 \\
10 \tdot2 & 5^2\tdot 2, 10\tdot 2\\
11\tdot 1& 11\tdot 1 \\
12 & 2^6, 4\tdot 2^4, 4^2\tdot 2^2, 4^3, 6\tdot 2^3, 6\tdot 4 \tdot2, 6^2, 8 \tdot2^2, 8\tdot 4, 12\\
\hline
\end{array}
\verynormalcols 
\end{array}
\]
\caption{\label{T:cat12part1}
Cycle structures of $\beta$ and $\gamma$ such that 
$(\id, \beta, \gamma; (12)) \in \ptp(n)$ for $n \le 12$.}
\end{table}

\begin{table}
\[
\begin{array}{ccc}
\verynarrowcols\begin{array}{|r|r|}
\hline
\tspacer n=13 & \\
\beta & \gamma\bspacer\\
\hline\tspacer
1^{13} & 1^{13}, 2\tdot1^{11}, 2^2\tdot1^{9}, 2^3\tdot1^7, 2^4\tdot1^5, 2^5\tdot1^3, 2^6\tdot1 \\
2^4\tdot1^5 & 4^2\tdot1^5, 4^2\tdot2\tdot1^3, 4^2\tdot2^2\tdot1\\
2^6\tdot1 & 4^3\tdot1 \\
3^3\tdot1^4 & 6\tdot3\tdot2\tdot1^2 \\
3^4\tdot1 & 3^4\tdot1, 6\tdot3^2\tdot1, 6^2\tdot1 \\
4^2\tdot1^5 & 8\tdot2\tdot1^3, 8\tdot2^2\tdot 1\\
4^2\tdot2^2\tdot1 & 8\tdot4\tdot1\\
5^2\tdot1^3 & 5^2\tdot1^3, 5^2\tdot2\tdot1, 10\tdot1^3, 10\tdot2\tdot1\\
6^2\tdot1 & 12\tdot1 \\
13 & 2^6\tdot 1,13\\
\hline
\hline
\tspacer n=14 & \\
\beta & \gamma\bspacer\\
\hline\tspacer
1^{14} & 1^{14}, 2\tdot1^{12}, 2^2\tdot1^{10}, 2^3\tdot1^8, 2^4\tdot1^6,
 2^5\tdot1^4, 2^6\tdot1^2 \\
2^4\tdot1^6 & 4^2\tdot1^6, 4^2\tdot2\tdot1^4, 4^2\tdot2^2\tdot1^2 \\
2^6\tdot1^2 & 4^3\tdot1^2\\
2^7 & 2^7, 4\tdot2^5, 4^2\tdot2^3, 4^3\tdot2 \\
3^3\tdot1^5 & 3^3\tdot 2\tdot 1^3,3^3\tdot2^2\tdot1,6\tdot 3\tdot 2\tdot 1^3, 6\tdot3\tdot2^2\tdot1 \\
3^4\tdot1^2 & 3^4\tdot1^2,6\tdot3^2\tdot1^2,6^2\tdot1^2 \\
4^2\tdot1^6 & 8\tdot2\tdot1^4, 8\tdot2^2\tdot1^2 \\
4^2\tdot2^2\tdot1^2 & 8\tdot4\tdot1^2\\
4^2\tdot2^3 & 8\tdot2^3, 8\tdot4\tdot2\\
5^2\tdot1^4 & 5^2\tdot1^4, 5^2\tdot2\tdot1^2, 10\tdot1^4, 10\tdot2\tdot1^2 \\
6^2\tdot1^2 & 12\tdot1^2\\
6^2\tdot2 & 6^2\tdot2, 12\tdot2 \\
7\tdot1^7 & 7\tdot2^3\tdot1 \\
7^2 & 2^3\tdot1^8, 2^4\tdot1^6, 2^5\tdot1^4, 2^6\tdot1^2, \\
&7\tdot1^7, 7\tdot2\tdot1^5, 7\tdot2^2\tdot1^3, 7\tdot2^3\tdot1, 7^2\\
9\tdot1^5 & 9\tdot2^2\tdot1 \\
10\tdot2^2 & 10\tdot2^2,10\tdot4\\
11\tdot1^3 & 11\tdot2\tdot1 \\
13\tdot1 & 13\tdot1 \\
14 &  2^7, 4\tdot2^5, 4^2\tdot2^3, 4^3\tdot2, 14 \\
\hline
\hline
\tspacer n=15 & \\
\beta & \gamma\bspacer\\
\hline\tspacer
1^{15}& 1^{15}, 2\tdot1^{13}, 2^2\tdot1^{11}, 2^3\tdot1^9,\\
 & 2^4\tdot1^7, 2^5\tdot1^5, 2^6\tdot1^3, 2^7\tdot1\\
2^4\tdot1^7& 4^2\tdot2\tdot1^5, 4^2\tdot2^2\tdot1^3, 4^2\tdot2^3\tdot1\\
2^6\tdot1^3& 4^3\tdot1^3, 4^3\tdot2\tdot1\\
3^4\tdot1^3& 3^4\tdot1^3, 3^4\tdot2\tdot1, 6\tdot3^2\tdot1^3, 6\tdot3^2\tdot2\tdot1, 6^2\tdot1^3, 6^2\tdot2\tdot1\\
3^5& 2^5\tdot1^5, 2^6\tdot1^3, 2^7\tdot1, 3\tdot2^4\tdot1^4, 3\tdot2^5\tdot1^2, 3\tdot2^6,\\ 
&3^2\tdot2^3\tdot1^3, 3^2\tdot2^4\tdot1, 3^3\tdot2^2\tdot1^2, 3^3\tdot2^3,  \\ 
& 3^4\tdot2\tdot1, 3^5, 6\tdot2^2\tdot1^5, 6\tdot2^3\tdot1^3, 6\tdot2^4\tdot1,\\ 
&6\tdot3\tdot2\tdot1^4, 6\tdot3\tdot2^2\tdot1^2, 6\tdot3\tdot2^3, 6\tdot3^2\tdot1^3, \\ 
 & 6\tdot3^2\tdot2\tdot1, 6\tdot3^3, 6^2\tdot1^3, 6^2\tdot2\tdot1, 6^2\tdot3\\
4^2\tdot1^7& 8\tdot2^2\tdot1^3, 8\tdot2^3\tdot1\\
5^2\tdot1^5& 5^2\tdot1^5, 5^2\tdot2\tdot1^3, 5^2\tdot2^2\tdot1,\\ 
 &10\tdot1^5, 10\tdot2\tdot1^3, 10\tdot2^2\tdot1\\
5^3& 2^6\tdot1^3, 2^7\tdot1, 5\tdot2^4\tdot1^2, 5\tdot2^5, 5^2\tdot2^2\tdot1, 5^3,\\ 
&10\tdot2\tdot1^3, 10\tdot2^2\tdot1, 10\tdot5\\
6^2\tdot1^3& 12\tdot1^3, 12\tdot2\tdot1\\
6^2\tdot3& 4^3\tdot2\tdot1, 4^3\tdot3, 12\tdot2\tdot1, 12\tdot3\\
7^2\tdot1& 7^2\tdot1, 14\tdot1\\
15& 2^7\tdot1, 3\tdot2^6, 5\tdot2^5, 6\tdot2^4\tdot1, 6\tdot3\tdot2^3, 6\tdot5\tdot2^2,\\ 
&6^2\tdot2\tdot1, 6^2\tdot3, 10\tdot2^2\tdot1, 10\tdot3\tdot2, 10\tdot5, 15\\
\hline
\end{array}
\verynormalcols 

&&

\verynarrowcols\begin{array}{|r|r|}
\hline
\tspacer n=16 & \\
\beta & \gamma\bspacer\\
\hline\tspacer
1^{16}& 1^{16}, 2\tdot1^{14}, 2^2\tdot1^{12}, 2^3\tdot1^{10}, 2^4\tdot1^8, 2^5\tdot1^6, 2^6\tdot1^4, 2^7\tdot1^2\\
2^4\tdot1^8& 4^2\tdot2\tdot1^6, 4^2\tdot2^2\tdot1^4, 4^2\tdot2^3\tdot1^2\\
2^6\tdot1^4& 4^3\tdot1^4, 4^3\tdot2\tdot1^2\\
2^8& 2\tdot1^{14}, 2^2\tdot1^{12}, 2^3\tdot1^{10}, 2^4\tdot1^8, 2^5\tdot1^6, 2^6\tdot1^4, 2^7\tdot1^2, 2^8,\\ 
 &4\tdot1^{12}, 4\tdot2\tdot1^{10}, 4\tdot2^2\tdot1^8, 4\tdot2^3\tdot1^6, 4\tdot2^4\tdot1^4, 4\tdot2^5\tdot1^2, 4\tdot2^6,\\ 
 &4^2\tdot1^8, 4^2\tdot2\tdot1^6, 4^2\tdot2^2\tdot1^4, 4^2\tdot2^3\tdot1^2, 4^2\tdot2^4,\\
 & 4^3\tdot1^4, 4^3\tdot2\tdot1^2, 4^3\tdot2^2, 4^4\\
3^3\tdot1^7& 3^3\tdot2^2\tdot1^3, 3^3\tdot2^3\tdot1, 6\tdot3\tdot2^2\tdot1^3, 6\tdot3\tdot2^3\tdot1\\
3^4\tdot1^4& 3^4\tdot1^4, 3^4\tdot2\tdot1^2, 6\tdot3^2\tdot1^4, 6\tdot3^2\tdot2\tdot1^2, 6^2\tdot1^4, 6^2\tdot2\tdot1^2\\
3^5\tdot1& 3^5\tdot1, 6\tdot3^3\tdot1, 6^2\tdot3\tdot1\\
4^2\tdot1^8& 8\tdot2^2\tdot1^4, 8\tdot2^3\tdot1^2\\
4^2\tdot2^4& 8\tdot2^2\tdot1^4, 8\tdot2^3\tdot1^2, 8\tdot2^4, 8\tdot4\tdot1^4, 8\tdot4\tdot2\tdot1^2, 8\tdot4\tdot2^2, 8\tdot4^2\\
4^4& 2^2\tdot1^{12}, 2^3\tdot1^{10}, 2^4\tdot1^8, 2^5\tdot1^6, 2^6\tdot1^4, 2^7\tdot1^2, 2^8,\\ 
&4\tdot1^{12}, 4\tdot2\tdot1^{10}, 4\tdot2^2\tdot1^8, 4\tdot2^3\tdot1^6, 4\tdot2^4\tdot1^4, 4\tdot2^5\tdot1^2, 4\tdot2^6,\\ 
&4^2\tdot1^8, 4^2\tdot2\tdot1^6, 4^2\tdot2^2\tdot1^4, 4^2\tdot2^3\tdot1^2, 4^2\tdot2^4,\\
 & 4^3\tdot1^4, 4^3\tdot2\tdot1^2, 4^3\tdot2^2, 4^4, 8\tdot1^8, 8\tdot2\tdot1^6, 8\tdot2^2\tdot1^4,\\ 
 &8\tdot2^3\tdot1^2, 8\tdot2^4, 8\tdot4\tdot1^4, 8\tdot4\tdot2\tdot1^2, 8\tdot4\tdot2^2, 8\tdot4^2, 8^2\\
5^2\tdot1^6& 5^2\tdot1^6, 5^2\tdot2\tdot1^4, 5^2\tdot2^2\tdot1^2, 10\tdot1^6, 10\tdot2\tdot1^4, 10\tdot2^2\tdot1^2\\
5^3\tdot1& 5^3\tdot1, 10\tdot5\tdot1\\
6^2\tdot1^4& 12\tdot1^4, 12\tdot2\tdot1^2\\
6^2\tdot2^2& 3^4\tdot2\tdot1^2, 3^4\tdot2^2, 4\tdot3^4, 6\tdot3^2\tdot2\tdot1^2, 6\tdot3^2\tdot2^2, 6\tdot4\tdot3^2,\\ 
 &6^2\tdot2\tdot1^2, 6^2\tdot2^2, 6^2\tdot4, 12\tdot2\tdot1^2, 12\tdot2^2, 12\tdot4\\
7^2\tdot1^2& 7^2\tdot1^2, 14\tdot1^2\\
8^2& 2^4\tdot1^8, 2^5\tdot1^6, 2^6\tdot1^4, 2^7\tdot1^2, 2^8, 4\tdot2^2\tdot1^8, 4\tdot2^3\tdot1^6, 4\tdot2^4\tdot1^4,\\
& 4\tdot2^5\tdot1^2, 4\tdot2^6, 4^2\tdot1^8, 4^2\tdot2\tdot1^6, 4^2\tdot2^2\tdot1^4, 4^2\tdot2^3\tdot1^2, 4^2\tdot2^4,\\ 
 &4^3\tdot1^4, 4^3\tdot2\tdot1^2, 4^3\tdot2^2, 4^4, 8\tdot1^8, 8\tdot2\tdot1^6, 8\tdot2^2\tdot1^4, 8\tdot2^3\tdot1^2, \\ 
 & 8\tdot2^4, 8\tdot4\tdot1^4, 8\tdot4\tdot2\tdot1^2, 8\tdot4\tdot2^2, 8\tdot4^2, 8^2, 16\\
10\tdot2^3& 5^2\tdot2^3, 5^2\tdot4\tdot2, 10\tdot2^3, 10\tdot4\tdot2\\
12\tdot4& 3^4\tdot2^2, 4\tdot3^4, 6\tdot3^2\tdot2^2, 6\tdot4\tdot3^2, 6^2\tdot2^2, 6^2\tdot4, 12\tdot2^2, 12\tdot4\\
14\tdot2& 7^2\tdot2, 14\tdot2\\
16& 2^8, 4\tdot2^6, 4^2\tdot2^4, 4^3\tdot2^2, 4^4, 8\tdot2^4, 8\tdot4\tdot2^2, 8\tdot4^2, 8^2, 16\\
\hline
\hline
\tspacer n=17 & \\
\beta & \gamma\bspacer\\
\hline\tspacer
1^{17}& 1^{17}, 2\tdot1^{15}, 2^2\tdot1^{13}, 2^3\tdot1^{11}, 2^4\tdot1^9, 2^5\tdot1^7, 2^6\tdot1^5, 2^7\tdot1^3, 2^8\tdot1\\
2^6\tdot1^5& 4^3\tdot1^5, 4^3\tdot2\tdot1^3, 4^3\tdot2^2\tdot1\\
3^4\tdot1^5& 3^4\tdot1^5, 3^4\tdot2\tdot1^3, 3^4\tdot2^2\tdot1,\\ 
 &6\tdot3^2\tdot1^5, 6\tdot3^2\tdot2\tdot1^3, 6\tdot3^2\tdot2^2\tdot1, 6^2\tdot1^5, 6^2\tdot2\tdot1^3, 6^2\tdot2^2\tdot1\\
3^5\tdot1^2& 6\tdot3^3\tdot1^2, 6^2\tdot3\tdot1^2\\
5^2\tdot1^7& 5^2\tdot2\tdot1^5, 5^2\tdot2^2\tdot1^3, 5^2\tdot2^3\tdot1, 10\tdot2\tdot1^5, 10\tdot2^2\tdot1^3, 10\tdot2^3\tdot1\\
6^2\tdot1^5& 12\tdot1^5, 12\tdot2\tdot1^3, 12\tdot2^2\tdot1\\
7^2\tdot1^3& 7^2\tdot1^3, 7^2\tdot2\tdot1, 14\tdot1^3, 14\tdot2\tdot1\\
17& 2^8\tdot1, 17\\
\hline
\end{array}
\verynormalcols 
\end{array}
\]
\caption{\label{T:cat12part2}
Cycle structures of $\beta$ and $\gamma$ such that 
$(\id, \beta, \gamma; (12)) \in \ptp(n)$ for $13\le n \le 17$.}
\end{table}

\begin{table}
\[
\begin{array}{ccccc}
\verynarrowcols\begin{array}{|r|}
\hline
\tspacer n=2  \\
  \gamma\bspacer\\
\hline\tspacer
  1^2 \\
\hline
\hline
\tspacer n=3  \\
  \gamma\bspacer\\
\hline\tspacer
  1^3, 2\tdot1 \\
\hline
\hline
\tspacer n=4  \\
  \gamma\bspacer\\
\hline\tspacer
  1^4, 2\tdot1^2, 2^2 \\
 \hline
 \hline
 \tspacer n=5  \\
  \gamma\bspacer\\
\hline\tspacer
  1^5, 2^2\tdot1, 4\tdot1 \\
 \hline
 \hline
 \tspacer n=6  \\
  \gamma\bspacer\\
\hline\tspacer
  1^6, 2^2\tdot1^2, 3\tdot1^3 \\
 \hline
 \hline
 \tspacer n=7  \\
  \gamma\bspacer\\
\hline\tspacer
  1^7, 2^2\tdot1^3, 2^3\tdot1, \\
  4\tdot1^3, 4\tdot2\tdot1, 5\tdot1^2, 7 \\
 \hline
 \hline
 \tspacer n=8  \\
  \gamma\bspacer\\
\hline\tspacer
  1^8, 2^2\tdot1^4, 2^3\tdot1^2, \\
  2^4, 4\tdot1^4, 4\tdot2\tdot1^2, \\
  4\tdot2^2,4^2, 7\tdot1 \\
 \hline
 \end{array}
 \verynormalcols

&&

\verynarrowcols\begin{array}{|r|}
 \hline
 \tspacer n=9  \\
  \gamma\bspacer\\
\hline\tspacer
  1^9, 2^3\tdot1^3, 2^4\tdot1, \\
  3^2\tdot 1^3, 3^3, 4^2\tdot1, 5\tdot1^4, \\
  6\tdot1^3, 6\tdot2\tdot1, 8\tdot1 \\
 \hline
 \hline
 \tspacer n=10  \\
  \gamma\bspacer\\
\hline\tspacer
  1^{10}, 2^3\tdot1^4, 2^4\tdot1^2, 3^3\tdot1,\\ 
 4^2\tdot1^2, 5\tdot1^5, 7\tdot1^3, 8\tdot1^2 \\
 \hline
 \hline
 \tspacer n=11  \\
  \gamma\bspacer\\
\hline\tspacer
 1^{11}, 2^3\tdot1^5, 2^4\tdot1^3, 2^5\tdot1, 3^3\tdot1^2, \\
 4^2\tdot1^3, 4^2\tdot2\tdot1, 5^2\tdot1, 7\tdot1^4, 10\tdot1\\
\hline
\hline
\tspacer n=12  \\
  \gamma\bspacer\\
\hline\tspacer
 1^{12}, 2^3\tdot1^6, 2^4\tdot1^4, 2^5\tdot1^2, 2^6, \\
 3^2\tdot1^6, 3^3\tdot1^3, 4^2\tdot1^4, 4^2\tdot2\tdot1^2, 4^2\tdot2^2, \\
 5^2\tdot1^2, 6\tdot1^6,6 \tdot2^2\tdot1^2, 6\tdot3\tdot1^3, \\
 8\tdot1^4,8\tdot2\tdot1^2, 8\tdot2^2, 9\tdot1^3 \\
\hline
\hline
\tspacer n=13  \\
  \gamma\bspacer\\
\hline\tspacer
 1^{13}, 2^4\tdot1^5, 2^5\tdot1^3, 2^6\tdot1, 3^3\tdot1^4, \\
 4^2\tdot1^5, 4^2\tdot2^2\tdot1, 4^3\tdot1, 5^2\tdot1^3, 7\tdot1^6, \\
 8\tdot1^5, 8\tdot2^2\tdot1, 8\tdot4\tdot1, 10\tdot1^3, 10\tdot2\tdot1, \\
 11\tdot1^2, 13 \\
\hline
\end{array}
\verynormalcols

&&

\verynarrowcols\begin{array}{|r|}
\hline
\tspacer n=14  \\
  \gamma\bspacer\\
\hline\tspacer
 1^{14}, 2^4\tdot1^6, 2^5\tdot1^4, 2^6\tdot1^2, 3^3\tdot1^5,\\
 4^2\tdot1^6, 4^2\tdot2^2\tdot1^2, 4^3\tdot1^2, 5^2\tdot1^4,\\ 
 7\tdot1^7, 7^2, 10\tdot1^4, 10\tdot2\tdot1^2, 13\tdot1 \\
\hline
\hline
\tspacer n=15  \\
  \gamma\bspacer\\
\hline\tspacer
1^{15},  2^4\tdot1^7,  2^5\tdot1^5,  2^6\tdot1^3,  2^7\tdot1, 3^3\tdot1^6,  3^4\tdot1^3, \\  
4^2\tdot1^7,  4^2\tdot2^2\tdot1^3,  4^2\tdot2^3\tdot1,  4^3\tdot1^3,  4^3\tdot2\tdot1,\\  
5^2\tdot1^5,  5^3,  6^2\tdot1^3,  6^2\tdot2\tdot1,  7^2\tdot1,\\  
8\tdot1^7,  8\tdot2^2\tdot1^3,  8\tdot2^3\tdot1,  8\tdot4\tdot1^3,  8\tdot4\tdot2\tdot1,\\  
9\tdot1^6,  9\tdot3\tdot1^3,  11\tdot1^4,  12\tdot1^3,  12\tdot2\tdot1,  14\tdot1\\
\hline
\hline
\tspacer n=16  \\
  \gamma\bspacer\\
\hline\tspacer
1^{16},  2^4\tdot1^8,  2^5\tdot1^6,  2^6\tdot1^4,  2^7\tdot1^2,  2^8, 3^3\tdot1^7,  \\  
4^2\tdot1^8,  4^2\tdot2^2\tdot1^4,  4^2\tdot2^3\tdot1^2,  4^2\tdot2^4,\\  
4^3\tdot1^4,  4^3\tdot2\tdot1^2,  4^3\tdot2^2,  4^4, 5^2\tdot1^6,  5^3\tdot1,\\  
6\tdot3\tdot2^2\tdot1^3,  7^2\tdot1^2, 8\tdot1^8,  8\tdot2^2\tdot1^4,  8\tdot2^3\tdot1^2,  8\tdot2^4,\\  
8\tdot4\tdot1^4,  8\tdot4\tdot2\tdot1^2,  8\tdot4\tdot2^2,  8\tdot4^2,  8^2,\\  
10\tdot1^6,  10\tdot2^2\tdot1^2,  11\tdot1^5,  13\tdot1^3,  14\tdot1^2\\
\hline
\hline
\tspacer n=17  \\
\gamma\bspacer\\
\hline\tspacer
1^{17},  2^5\tdot1^7,  2^6\tdot1^5,  2^7\tdot1^3,  2^8\tdot1, 3^3\tdot1^8,  \\  
4^3\tdot1^5,  4^3\tdot2^2\tdot1,  4^4\tdot1,  5^2\tdot1^7,  5^3\tdot1^2, 7^2\tdot1^3,  8^2\tdot1,\\  
10\tdot1^7,  10\tdot2^2\tdot1^3,  10\tdot2^3\tdot1,  10\tdot5\tdot1^2,  13\tdot1^4, 16\tdot1\\
\hline
\end{array}
\verynormalcols 
\end{array}
\]
\caption{\label{T:cat123}Cycle structures of $\gamma$ such that 
$(\id, \id, \gamma; (123)) \in \ptp(n)$ for $n \le 17$. 
}
\end{table}

We have proved several general necessary conditions for
$(\id,\id,\gamma;(123))$ to be in $\ptp(n)$, and provided 
complete characterisations of some simple cases. 
It is time to bring all our results from this section
and the previous one together.

\section{Small orders and asymptotics}\label{s:end}

In this final section we tie the earlier threads together. We apply
the theory we have developed to two opposite extremes, exhaustively
checking small orders before looking at some asymptotic trends.

We first describe how we used the preceding results to establish exactly
what $\ptp(n)$ is when $n\le17$. For each possible cycle structure of
$\beta$ and $\gamma$ we first considered whether any of our results
showed that $(\id,\beta,\gamma;(12))\notin\ptp(n)$ or
$(\id,\id,\gamma;(123))\notin\ptp(n)$. 
When applying \lref{l:lcm12} and \lref{l:lcm123} we checked for each
block that there were sufficient symbols available to fill it.
When applying \tref{t:strlcm12} and \tref{t:strlcm123} we chose
$\Lambda$ to be the set of divisors of the length of some cycle of
$\beta$ and $\gamma$, respectively. This guaranteed that we would
find a (not necessarily proper) subsquare, of order say $s$.
If $n/2<s<n$ this is an immediate contradiction. If $s\le n/2$
the subsquare has an induced autoparatopism, and we checked
with a recursive call that it was plausible. If $s=n/2$ we also
considered the complementary subsquare, as described in the 
example after \tref{t:strlcm12}.

If none of our results precluded a particular autoparatopism, then we
attempted to construct a Latin square with that autoparatopism.  We
used the explicit constructions given in the proofs of \lref{l:totsym}
and \tref{t:betaid12} in cases where these results applied.  Also, if
\cref{cy:fixsbsq12} or \cref{cy:fixsbsq123} implied the existence of a
subsquare, we built that subsquare first.  With the caveats just
mentioned, a simple backtracking algorithm was quickly able to
construct a Latin square with the desired autoparatopism in all the
required cases. The resulting Latin squares can be downloaded from
\cite{WWWW}.

Thus, by combining \lref{l:lcm12} and Theorems \ref{t:strlcm12},
\ref{t:ord12}, \ref{t:cycdiv12}, \ref{t:fixgameven12},
\ref{t:fixgamodd12} and \ref{t:evencyc12} we found a catalogue of all
possible cycle structures for $(\id,\beta,\gamma;(12))\in\ptp(n)$ for
$n\le17$.  The results are given in \Tref{T:cat12part1} and
\Tref{T:cat12part2}.

Similarly, by combining \lref{l:lcm123} and \tref{t:strlcm123}
with the results in \sref{s:123}
we found a catalogue of all possible cycle structures for
$(\id,\id,\gamma;(123))\in\ptp(n)$ for $n\le17$.
The results are given in \Tref{T:cat123}. 
By \lref{l:conjauto} and \tref{t:conj} it is possible to deduce
from Tables \ref{T:cat12part1}, \ref{T:cat12part2} and \ref{T:cat123} 
a list of all 
$(\alpha,\beta,\gamma;\delta)\in\ptp(n)$ for $n\le17$, where
$\delta\ne\id$. The $\delta=\id$ case was already solved 
in \cite{Stones}.

\bigskip

We end with some interesting comparisons with the following theorem on
autotopisms by McKay {\em et al.} \cite{MWZ15}. In it and the
subsequent results, the phrase ``almost all'' refers to the asymptotic
proportion as $n\rightarrow\infty$.

\begin{theo}
  For almost all $\alpha \in \sym_n$, there are no
  $\beta,\gamma\in\sym_n$ such that $(\alpha,\beta,\gamma)\in\atp(n)$.
\end{theo}

In the same vein we have:

\begin{theo}\label{t:asy12}
  For almost all $\gamma \in \sym_n$, there are no
  $\alpha,\beta\in\sym_n$ such that $(\alpha,\beta,\gamma;(12))\in\ptp(n)$.
\end{theo}

\begin{proof}
If $\sigma=(\alpha,\beta,\gamma;(12))\in\ptp(n)$ then
$\sigma^2=(\alpha\beta,\beta\alpha,\gamma^2;\id)\in\ptp(n)$. In turn
this implies that $\gamma^2$ has order at most $n^2/4$, by \cite[Thm 2]{MWZ15}, 
so $\gamma$ has order at most $n^2/2$. However by \cite{ET67}, almost all
$\gamma\in\sym_n$ have order at least $n^{(1/2+o(1))\log n}$, from which the
result follows.
\end{proof}

\begin{coro}
  Almost all $\sigma\in\grp_n$ satisfy $\sigma\notin\ptp(n)$.
\end{coro}

\begin{proof}
  Let $\sigma=(\alpha,\beta,\gamma;\delta)$ be chosen uniformly at random
  from $\grp_n$. In light of \tref{t:conj},
  \tref{t:asy12} implies
  the result if $\delta$ is a $2$-cycle, and \cite{MWZ15} showed
  the case when $\delta=\id$. So it suffices to assume that $\delta=(123)$.
  If $\sigma\in\ptp(n)$ then 
  $\sigma^3=(\alpha\beta\gamma,\beta\gamma\alpha,\gamma\alpha\beta;\id)\in\atp(n)$.
  However, the cycle structure of $\alpha\beta\gamma$ has the same distribution
  as for a random permutation. Hence, by \cite{MWZ15}, the probability that
  $\sigma^3$ is an autotopism approaches $0$ as $n\rightarrow\infty$.
\end{proof}

These results contrast starkly with our final two observations:

\begin{theo}
  For all $\alpha \in \sym_n$ there exist $\beta, \gamma \in \sym_n$
  such that $\sigma=(\alpha, \beta, \gamma;(12)) \in \ptp(n)$.
\end{theo}

\begin{proof}
  For $\alpha \in \sym_n$ take $\gamma=(1\,2\cdots n)$ and
  $\beta=\alpha^{-1}\gamma$. Then 
  $\sigma'=(\id,\alpha\beta,\gamma;(12)) \in \ptp(n)$, by \cref{cj:1cyc12}. 
  Hence $\sigma\in\ptp(n)$, since $\sigma$ and $\sigma'$ are conjugate by
  \tref{t:conj}.
\end{proof}

\begin{theo}
For all $\alpha, \beta \in \sym_n$ there exist $\gamma \in \sym_n$ such 
that $\sigma=(\alpha, \beta, \gamma;(123)) \in \ptp(n)$.
\end{theo}

\begin{proof}
For $\alpha,\beta\in\sym_n$ take $\gamma=(\alpha\beta)^{-1}$. Then
 $\alpha\beta\gamma=\id$ and hence $\sigma=(\alpha, \beta,
 \gamma;(123))$ and $\sigma'=(\id, \id, \id;(123))$ are
 conjugate. Now apply \lref{l:totsym}.
\end{proof}

  \let\oldthebibliography=\thebibliography
  \let\endoldthebibliography=\endthebibliography
  \renewenvironment{thebibliography}[1]{%
    \begin{oldthebibliography}{#1}%
      \setlength{\parskip}{0.4ex plus 0.1ex minus 0.1ex}%
      \setlength{\itemsep}{0.4ex plus 0.1ex minus 0.1ex}%
  }%
  {%
    \end{oldthebibliography}%
  }

\end{document}